\documentclass[12pt,a4paper]{amsart} 
\usepackage[utf8]{inputenc}
\usepackage[english]{babel}
\usepackage{amsmath}
\usepackage{amsfonts}
\usepackage{amssymb}
\usepackage{mathrsfs}
\usepackage{color}
\usepackage{amsthm}
\newtheorem{prop}{Proposition} [section]

\newtheorem{prb}{Problem} [section]

\newtheorem{thm}{Theorem} [section]

\newtheorem*{remark}{Remark}

\usepackage{graphics}
\usepackage{epstopdf}
\usepackage{epsfig}
\usepackage{url}
\usepackage{hyperref}
\usepackage[a4paper,top=2.5cm,bottom=2.8cm,
left=3.2cm,right=3.2cm]{geometry}
\newcommand\De{\delta}
\newcommand\G{\gamma}

\newcommand\CB{\mathcal{B}}

\newcommand\CH{\mathcal{H}}

\newcommand\CT{\mathcal{T}}

\newcommand\Ov{\overrightarrow}
\newcommand\Ol{\overline}

\newcounter{num} 
\newcommand{\Fg}[1][]{\thenum}

\definecolor{pink}{rgb}{1.0,.4,.0} %gold \G(n)
\definecolor{green2}{rgb}{.7,.7,.7} % Ep2
\definecolor{pink2}{rgb}{1.,.8,.1} %7 pink Ep4
\definecolor{blue2}{rgb}{.6,.5,1} % Ep5
\definecolor{blue3}{rgb}{.4,.6,.2} %Ep6
\definecolor{pink3}{rgb}{.7,.4,.4} %5 Ep3
\definecolor{mag}{rgb}{.6,.0,1.0} % \G
\definecolor{rd}{rgb}{1.,.0,.0} % \Al

\begin{document}
% ===================
%Sangaku Journal of Mathematics (SJM) \copyright SJM \\
%ISSN 2534-9562 \\
%Volume 2 (2018), pp. 3-12  \\
%Received 8 February 2018. Published on-line 23 February 2018 \\ 
%web: \url{http://www.sangaku-journal.eu/} \\
%\copyright The Author(s) This article is published 
%with open access\footnote{This article is distributed under the terms of the Creative Commons Attribution License which permits any use, distribution, and reproduction in any medium, provided the original author(s) and the source are credited.}. \\
% ===========================   
\bigskip
\bigskip\bigskip

\begin{center}
{\Large \textbf{Haga's theorems in paper folding and related theorems in 
Wasan geometry Part 2}} \\
\bigskip

\bigskip
\textsc{Hiroshi Okumura} \\
%Maebashi Gunma 371-0123, Japan \\
%e-mail: \href{mailto:hokmr@protonmail.com}{hokmr@protonmail.com} \\

\bigskip\bigskip
\end{center}

\medskip

\textbf{Abstract.} 
We generalize problems in Wasan geometry which involve no folded figures 
but are related to Haga's fold in origami. Using the tangent circles 
appeared in those problems we give a parametric representation of the 
generalized Haga's fold given in the first part of this two-part paper. 



\bigskip
\bigskip

\section{Introduction}

This is the second part this two-part paper. In the first part we 
considered the generalized Haga's fold. There are several problems 
in Wasan geometry, which do not involve folded figures but are 
closely related Haga's fold. In this second part we consider 
those problems in a general way. Using tangent circles appeared in 
those problems, we give a parametric representation of the 
generalized Haga's fold.

\section{Related problems in Wasam geometry}\label{srw}

In this section we consider several problems in Wasan geometry closely 
related to Haga's fold, though they are not involving folded figures. 
A general solution of the problems is given in the next section. 
We start with two similar problems. The following problem can be found 
in \cite{Furu}, \cite{Toyo1,Toyo4}, \cite{sd1678} and \cite{tz1322}. 

\begin{prb} \label{ptsh} {\rm 
Let $\De$ be a circle of radius $d$ and let $ABCD$ be a rectangle 
sharing its center with $\De$, where the side $AB$ touches $\De$ and 
the side $BC$ intersect $\De$ in two points $($see Figure \ref{ftsh}$)$. 
The inradius of the curvilinear triangle made by $AB$, $BC$ and $\G$ 
is $r$ and the circle touching $BC$ at its midpoint and touching 
the minor arc of $\De$ cut by $BC$ has radius $r$. Find $d$ in terms  
of $r$ (or find $r$ in terms of $d$). } 
\end{prb}

\medskip
\begin{minipage}{0.5\hsize} %%%%%%%%%%%%%%%%%%%%%%%%%
\begin{center} 
\includegraphics[clip,width=40mm]{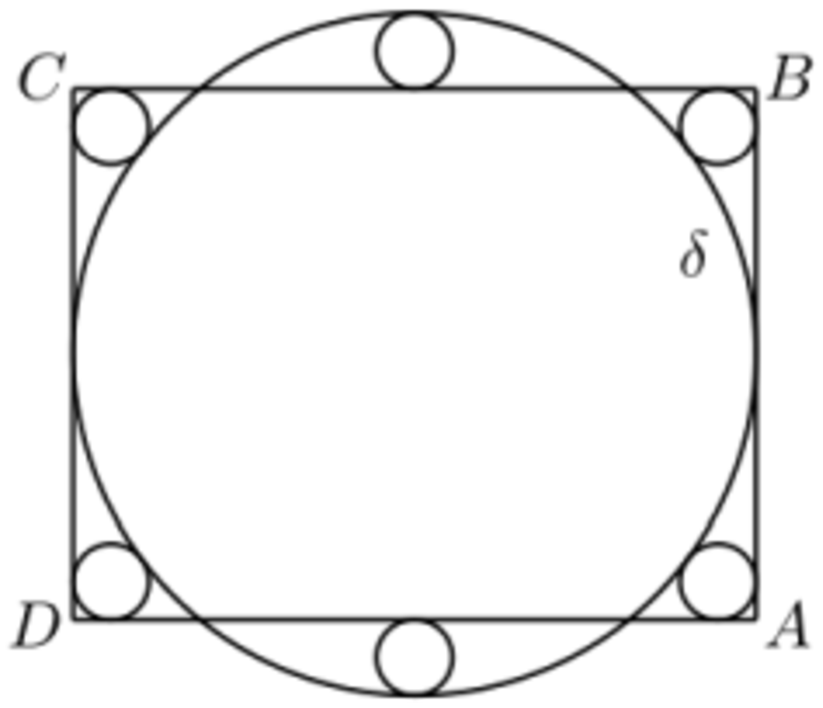}\refstepcounter{num}\label{ftsh}\\
Figure \Fg 
\end{center}  
\end{minipage}
\begin{minipage}{0.5\hsize}
\begin{center} 
%\vskip2mm
\includegraphics[clip,width=48mm]{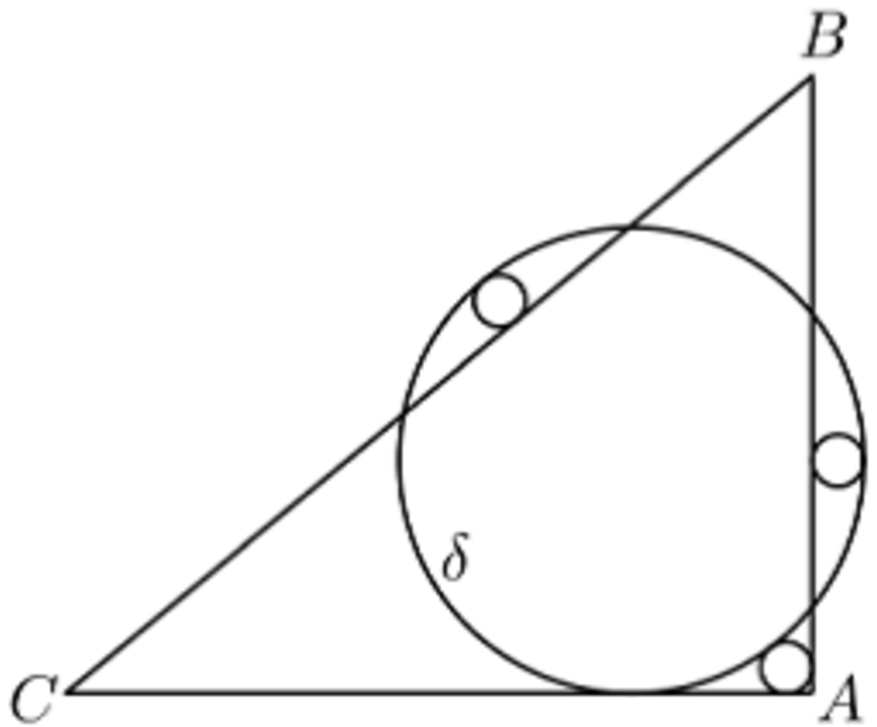}\refstepcounter{num}\label{fg}\\
%\input g.tex\refstepcounter{num}\label{fg}\\
%\vskip2mm
Figure \Fg 
\end{center} 
\end{minipage}  %%%%%%%%%%%%%%%%%%%%%%%%%%%%%%%%%%%%%

\medskip

\begin{prb}\label{pg} {\rm 
Let $\De$ be a circle of radius $d$ and let $ABC$ be a right triangle 
with right angle at $A$. The side $CA$ touches $\De$, and each of the 
sides $AB$ and $BC$ intersects $\De$ in two points. The inradius of 
the curvilinear triangle made by $CA$, $AB$ and $\De$ equals $r$. The 
maximal circle touching $AB$ from the side opposite to $C$ and 
touching $\De$ internally, and the maximal circle touching $BC$ from 
the side opposite to $A$ and touching $\De$ internally have radius $r$. 
Find $r$ in terms of $d$. }
\end{prb}

\medskip\medskip
\begin{minipage}{.38 \hsize} 
\begin{center} %111%%%%%%%%%%%%%%%%%%%%%
\includegraphics[clip,width=40mm]{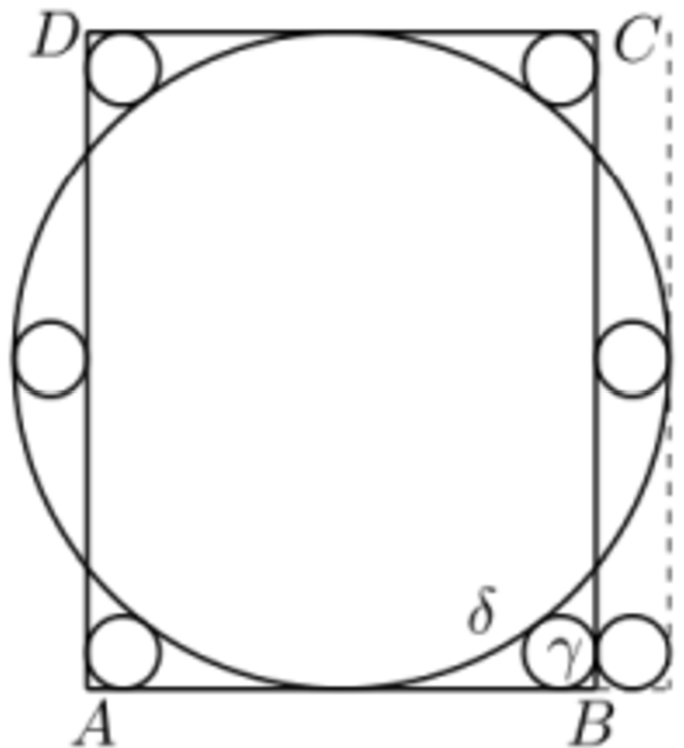}\refstepcounter{num}\label{ftsrev}\\
\vskip1mm
Figure \Fg 
\end{center} %%%%%%%%%%%%%%%%%%%%%%%%%
\end{minipage}
\begin{minipage}{.38 \hsize}
\begin{center}%222%%%%%%%%%%%
\vskip2.2mm
\includegraphics[clip,width=48mm]{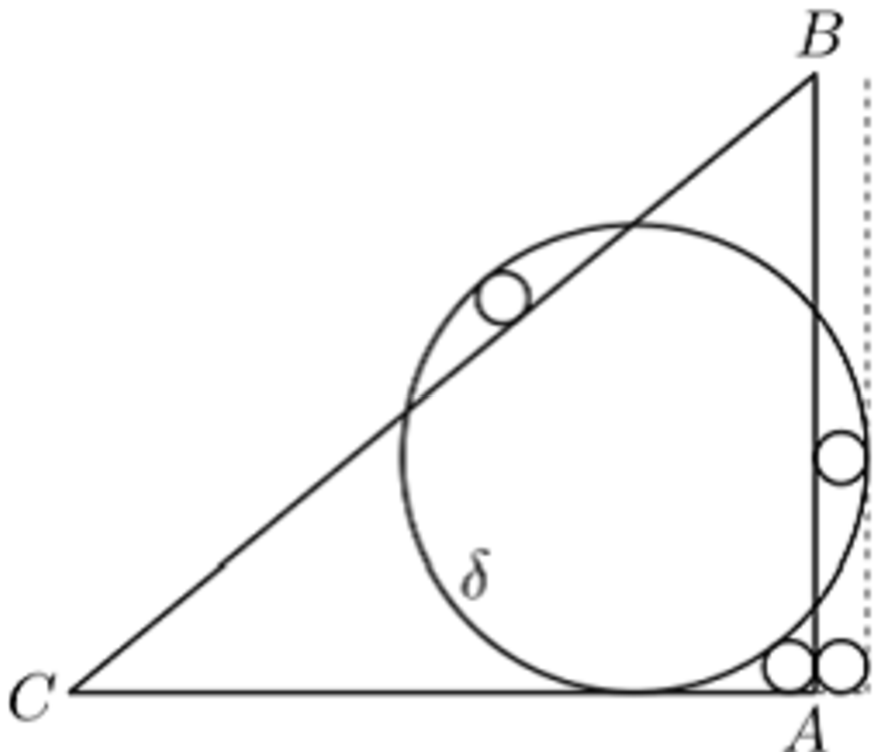}\refstepcounter{num}\label{fgrev}\\
\vskip1mm
Figure \Fg 
\end{center} %333%%%%%%%%%%%%%%%%% 
\end{minipage}
\begin{minipage}{.24 \hsize}
\begin{center} %%%%%%%%%%%%%%%%%%
\vskip10.5mm
\includegraphics[clip,width=28mm]{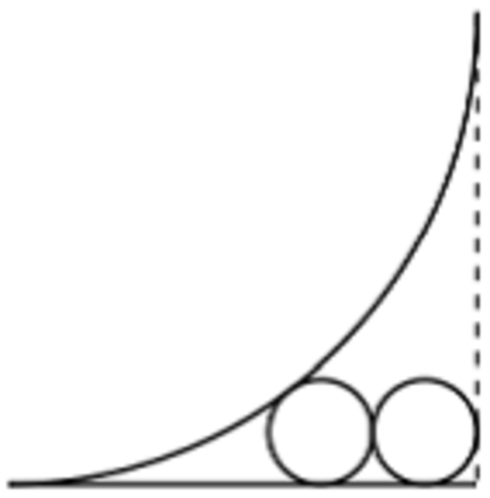}\refstepcounter{num}\label{fes}\\
\vskip.9mm
Figure \Fg 
\end{center} %%%%%%%%%%%%%%%%%%%%%%%%%%%
\end{minipage} 
\medskip

We show that the two problems are essentially the same. 
Let $\G$ be the incircle of the curvilinear triangle made by $AB$, $BC$ 
and $\De$ in Problem \ref{ptsh} (see Figure \ref{ftsrev}). If we draw the 
line parallel to $BC$ touching $\De$ and the reflection of $\G$ in the 
line $BC$ and extend the side $AB$, we get Figure \ref{fes}. 
We can also get the same figure from Figure \ref{fg} in a similar way 
(see Figure \ref{fgrev}). Therefore the two problems are essentially the 
same. Problems with Figure \ref{fes} can also be found in \cite{Ishida}, 
\cite{ito}, \cite{Toyo1,Toyo4}, \cite{zzkiou}, \cite{ets}, \cite{kan}, 
\cite{ssk} and \cite{uk}. 
A generalization of Problem \ref{ptsh} can be found in \cite{OKMIQ94}.

We state Problems \ref{ptsh} and \ref{pg} so that the body text gives 
enough information without the figures. However the most informations of 
the problems in Wasan geometry are given by the figures, thereby the body 
texts play only subsidiary roles. The next sangaku problem is stated in 
such a way \cite{gunma}: 
%Gunma 215 

\medskip
\begin{minipage}{.5 \hsize}%%%%%%%%%%%%%%%%%%%%%%%%%
\begin{center}
\vskip 1mm
\includegraphics[clip,width=48mm]{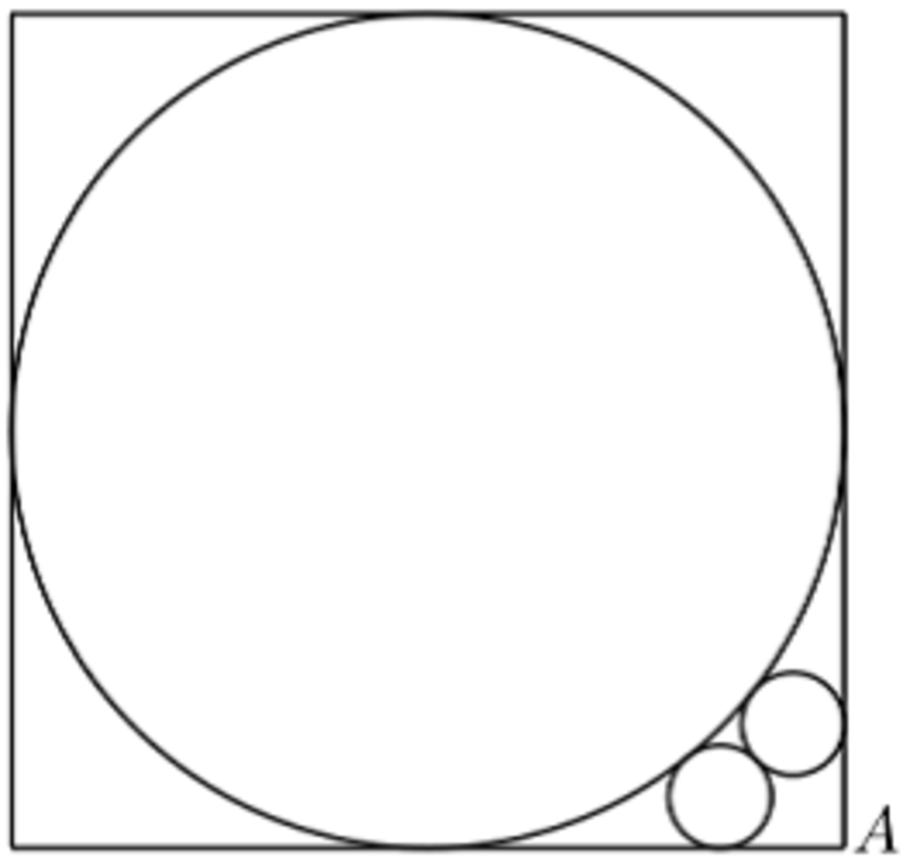}\refstepcounter{num}\label{fsantai}\\
\vskip 1mm
Figure \Fg . 
\end{center}%%%%%%%%%%%%%%%%%%%%%%%%%
\end{minipage}
\begin{minipage}{.5 \hsize}
\begin{center}%%%%%%%%%%%%%%%%%%%%%%%%%
\includegraphics[clip,width=48mm]{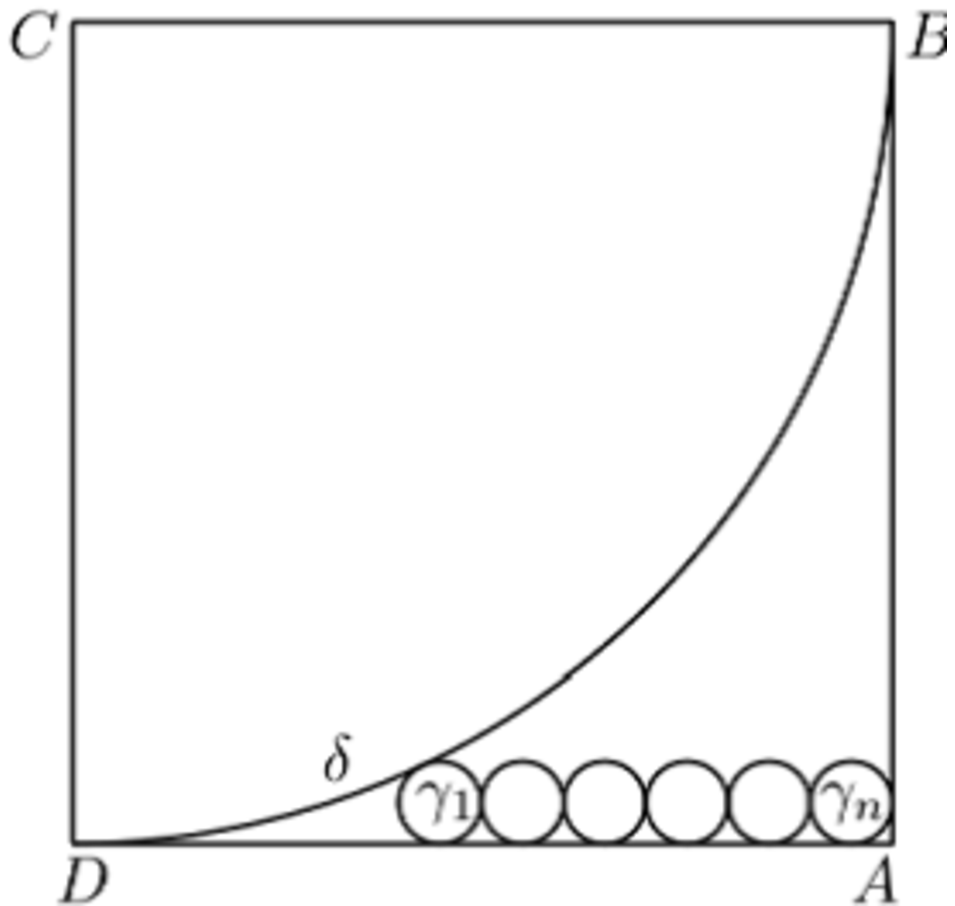}\refstepcounter{num}\label{f6}\\
%\input 6.tex\refstepcounter{num}\label{f6} \\
%\medskip
Figure \Fg . 
\end{center}
\end{minipage}%%%%%%%%%%%%%%%%%%%%%%%%%%%%%%%%%%%%%%%%%%%%%%%%%%%%%%

\begin{prb}\label{psantai} {\rm 
There are a large circle of radius $d$ and two small circles of radius 
$r$ in a square as in Figure \ref{fsantai}. Show $r$ in terms of $d$.
}\end{prb}

We show the problem is incorrect using the next proposition. 

\begin{prop}\label{p1}
If an external common tangent of externally touching 
two circles of radii $r$ and $s$ touches the circles at points $P$ 
and $Q$, then $|PQ|=2\sqrt{rs}$. 
\end{prop}

\medskip
The answer says $r=d/9$. But we have $d=2\sqrt{rd}+\sqrt{2}r+r$ by the 
proposition. Hence we get 
$d=\left(3+\sqrt{2}+2\sqrt{2+\sqrt{2}}\right)r\approx 8.11r$. Therefore 
the assertion of the problem is incorrect. It seem that the two small 
circles were described as in Figure \ref{fes} in the original problem, 
however Figure \ref{fsantai} was used by transcription error. 
%A problem with Figure \ref{fsantai} can also be found in \cite{uk}. 
A general case was considered by Toyoyoshi (see Figure \ref{f6}): 

\begin{prb}[\cite{Toyo2}]\label{pbt} {\rm
Let $\De$ be a circle of radius $d$ with center $C$ passing through $B$ 
for a square $ABCD$. Let $\G_1$, $\G_2$, $\cdots$, $\G_n$ be congruent 
circles of radius $r$ touching $DA$ from the same side such that $\G_1$ 
and $\G_2$ touch and $\G_i$ $(i=3,4,\cdots n)$ touches $\G_{i-1}$ from 
the side opposite to $\G_1$, also $\G_1$ touches $\De$ externally and 
$\G_n$ touches $AB$ from the same side as $\De$. Show $r$ in terms of 
 $s$ and $n$. }
\end{prb}

%%%%%%%%%%%%%%%%%%%%%%%%%%%%%%%%%%%%%%%%%%%%%%%%%%%%%%%%%%%%%%%%%%%%%%%%
\section{Generalized figure}\label{sgc}

We generalize the figures of Problems \ref{pbt}. Let $k$ and $l$ be 
perpendicular lines intersecting in a point $A$ and let $\G$ be a circle 
of radius $r$ touching $k$ at a point $K$.  Let $\De_1$ and $\De_2$ be circles 
of radii $d_1$ and $d_2$ ($d_2\le d_1$), respectively, touching $l$ from the same 
side, such that they touch $k$ from the same side as $\G$, and also touch 
$\G$ externally, where if $\G$ touches $k$ at $A$, we consider $\De_1$ is 
the circle touching $k$, $l$ and $\G$ externally, $\De_2=A$ and $d_2=0$. We 
denote the figure of $k$, $l$, $\G$, $\De_1$ and $\De_2$ by $\CT(n)$, where 
\begin{equation}\label{eqndef}
n=\frac{\tau|AK|}{2r}+\frac{1}{2}
\end{equation} 
and $\tau=1$ if $\De_1$ and $K$ lies on the same side of $l$ otherwise 
$\tau=-1$ (see Figures \ref{fdefcn} and \ref{fminogasi}). We also consider the 
case in which $\G$ degenerates to a point $K\not=A$ on $k$. We consider that  
$\De_1$ and $\De_2$ coincide and touch $k$ at $K$ in this case, and figure is 
denoted by $\CT\left(\Ol{0}\right)$ (see Figure \ref{fdefinf}).
We define the value of $\Ol{0}$ equals $0$. 

\medskip
\begin{minipage}{.5\hsize}%%%%%%%%%%%%%%%%%%%%%%%%%%%%
\begin{center} 
\includegraphics[clip,width=48mm]{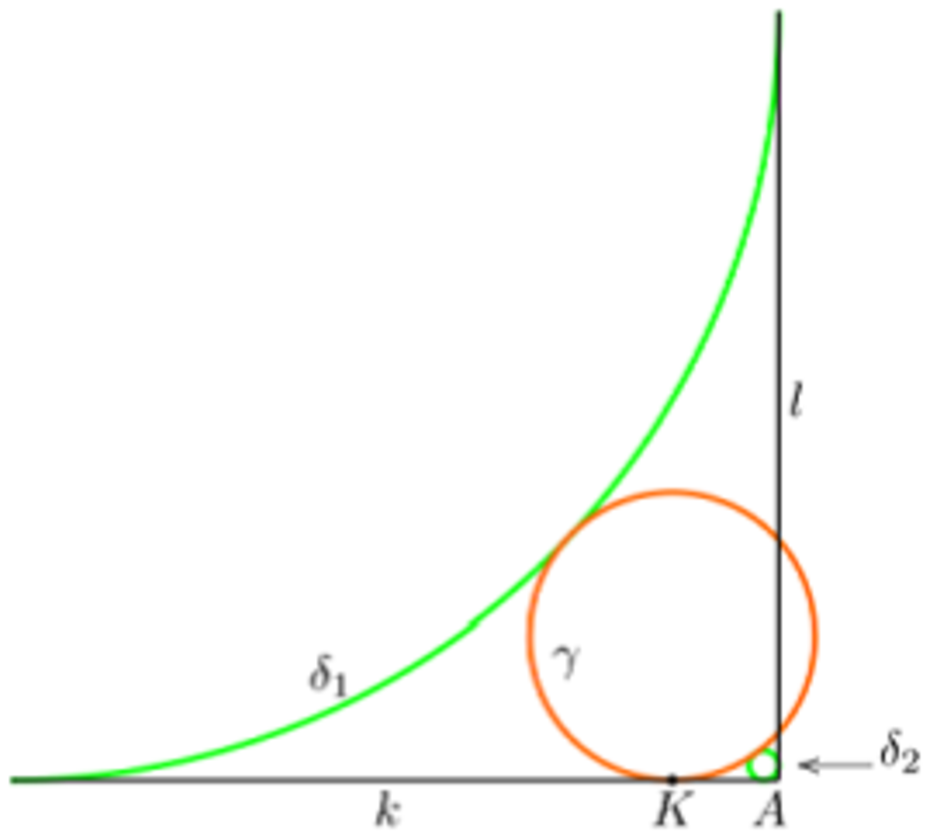}\refstepcounter{num}\label{fdefcn}\\
\medskip
Figure \Fg : $\tau=1$ $(1/2<n)$
\end{center}  %%%%%%%%%%%%%%%%%%%%%%%%%@
\end{minipage}
\begin{minipage}{.5\hsize}
\begin{center}%%%%%%%%%%%%%%%%%%%%%%%%%
\vskip8.5mm
\includegraphics[clip,width=48mm]{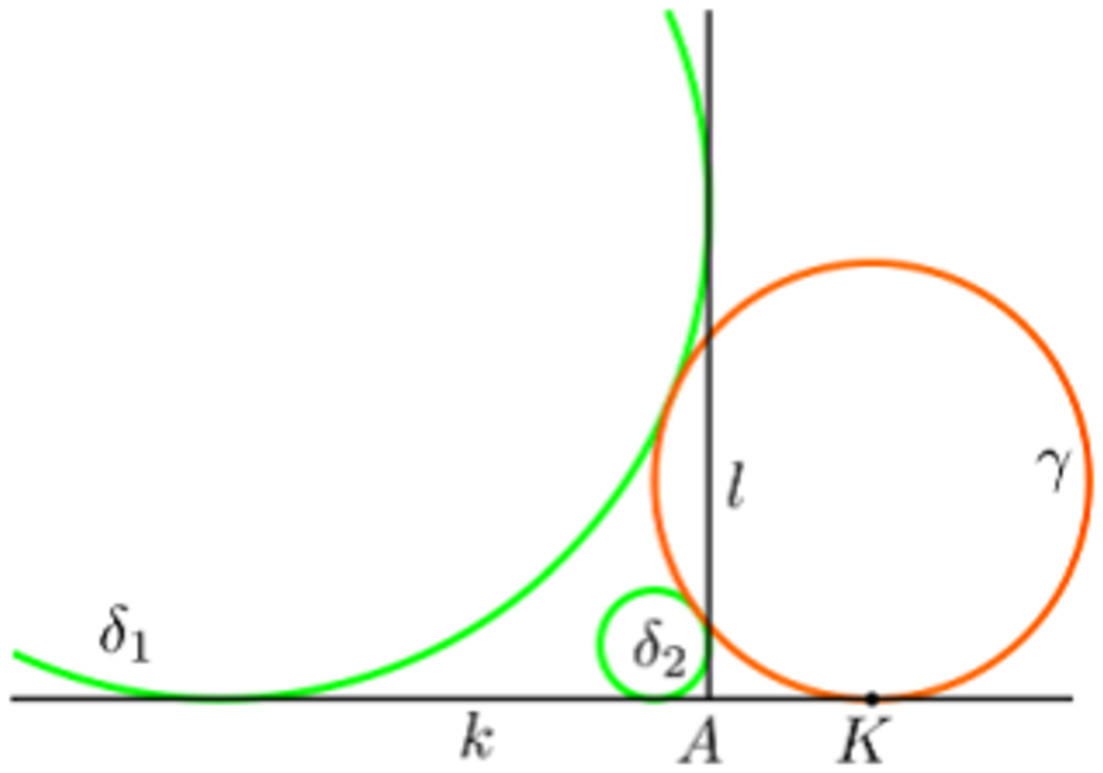}\refstepcounter{num}\label{fminogasi}\\
\vskip.7mm
Figure \Fg : $\tau=-1$ $(0<n<1/2)$
\end{center}
\end{minipage}%%%%%%%%%%%%%%%%%%%%%%%%%%%%%%%%%%%%%%%%%%%%
\medskip

\begin{minipage}{.26\hsize}%%%%%%%%%%%%%
\begin{center} 
\vskip2mm 
\includegraphics[clip,width=30mm]{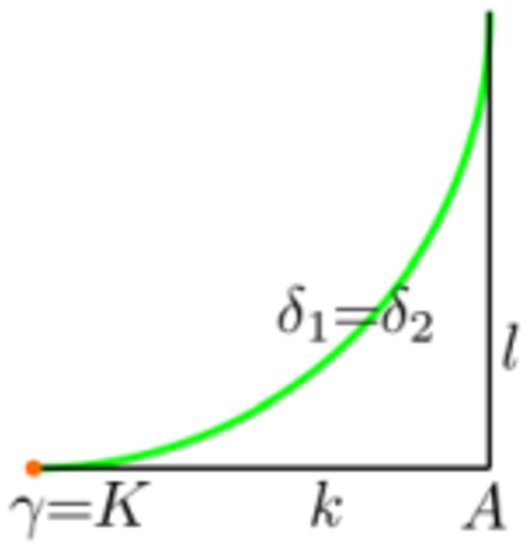}\refstepcounter{num}\label{fdefinf}\\
\vskip0.5mm
Figure \Fg : $\CT\left(\Ol{0}\right)$
\end{center}  %%%%%%%%%%%%%%%%%%%%%%%%%@
\end{minipage}%%%%%%%%%%%%%%%%%%%%%%%%%%%%%%%%%%%%%%%%%%%%%%%%%%%
\begin{minipage}{.29\hsize}%%%%%%%%%%%%%%%%%%%%%%%%%%%%%%%%%%%%%%%%%%
\begin{center}
\vskip.9mm
\includegraphics[clip,width=35mm]{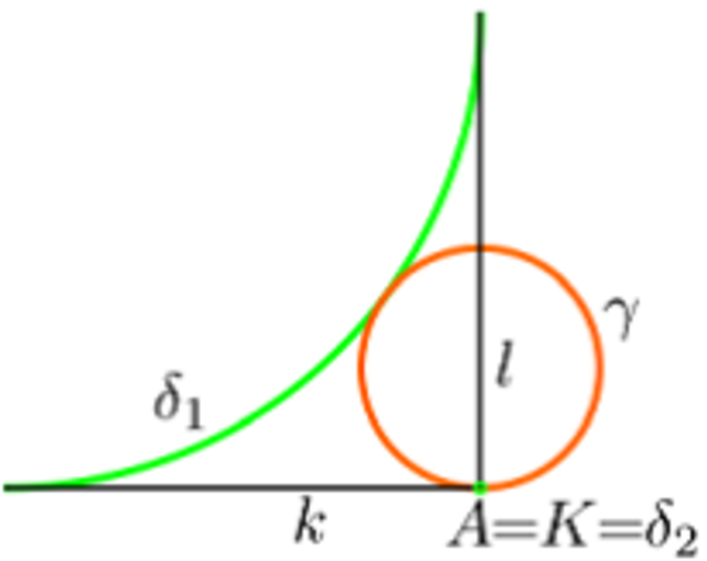}\refstepcounter{num}\label{fdefhalf}\\
\vskip0.4mm
Figure \Fg : $\CT(1/2)$
\end{center}
\end{minipage}%%%%%%%%%%%%%%%%%%%%%%%%%%%%%%%%%%%%%%%%%%
\begin{minipage}{.45\hsize}
\begin{center}%%%%%%%%%%%%%%%%%%%%%%%%%
\vskip1.0mm
\includegraphics[clip,width=48mm]{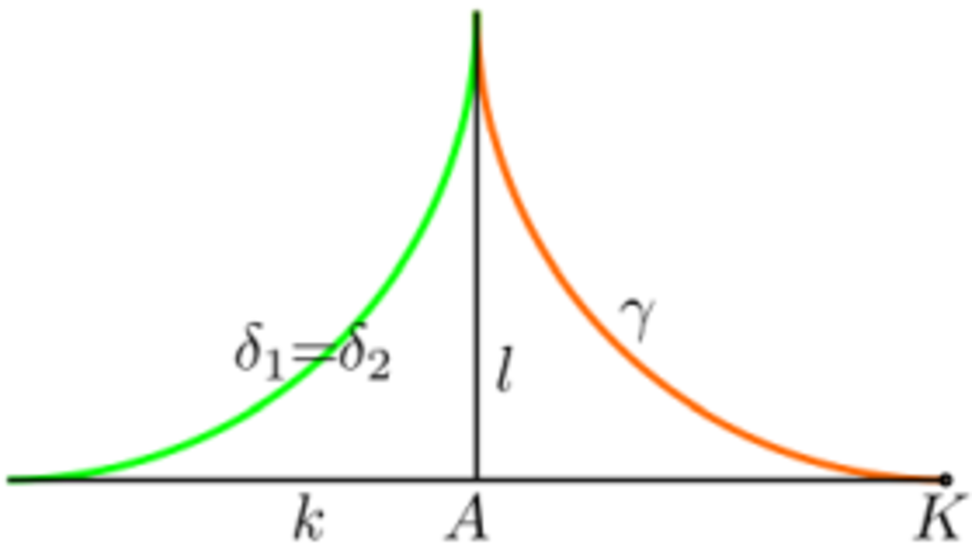}\refstepcounter{num}\label{fdef0}\\
\vskip.2mm
Figure \Fg : $\CT(0)$
\end{center}
\end{minipage}%%%%%%%%%%%%%%%%%%%%%%%%%
\medskip

\begin{minipage}{.5\hsize}%%%%%%%%%%%%%%%%%%%%%%%%%
\begin{center} 
\includegraphics[clip,width=48mm]{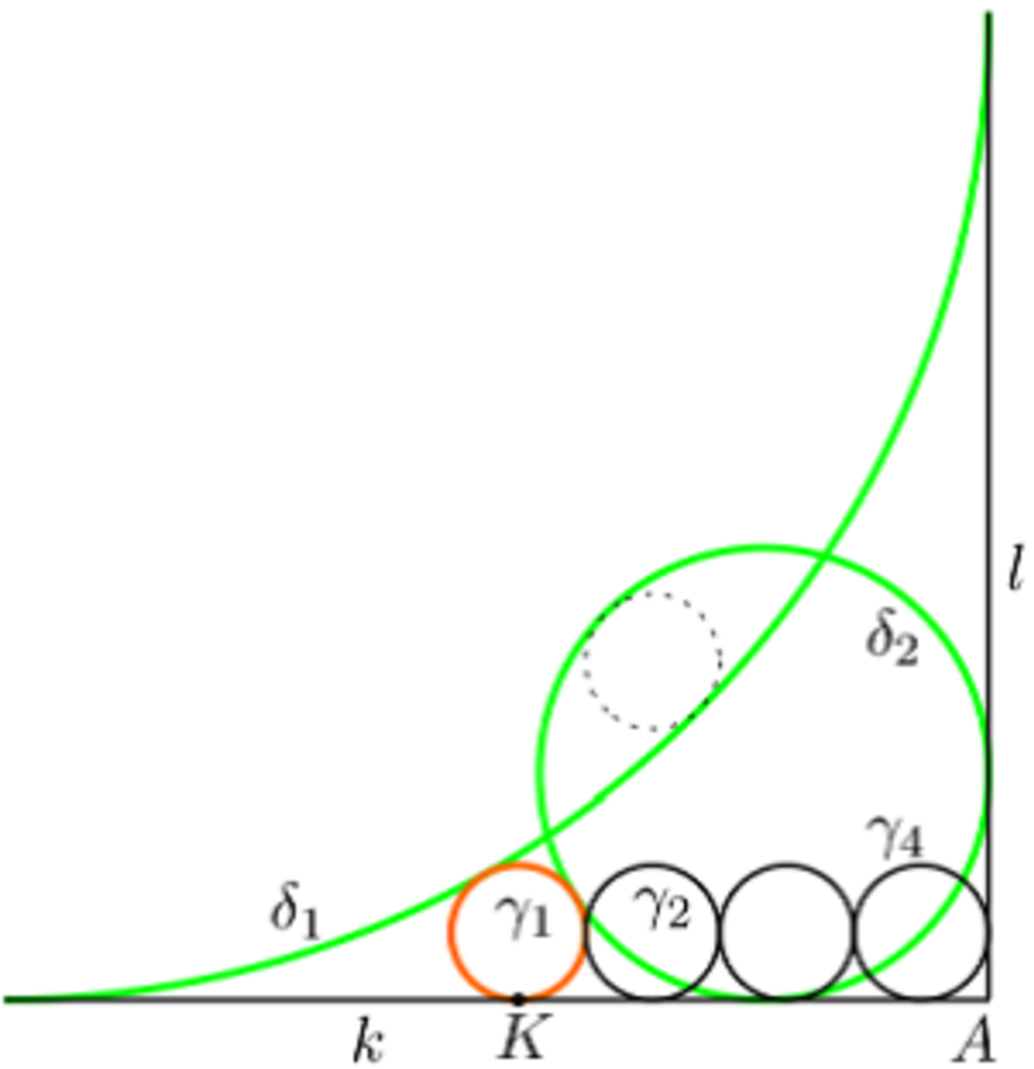}\refstepcounter{num}\label{f4}\\
\medskip\medskip
Figure \Fg : $\CT(4)$
\end{center}  %%%%%%%%%%%%%%%%%%%%%%%%%@
\end{minipage}
\begin{minipage}{.5\hsize}%%%%%%%%%%%%%%%%%%%%%%%%%
\begin{center} 
\vskip3mm
\includegraphics[clip,width=48mm]{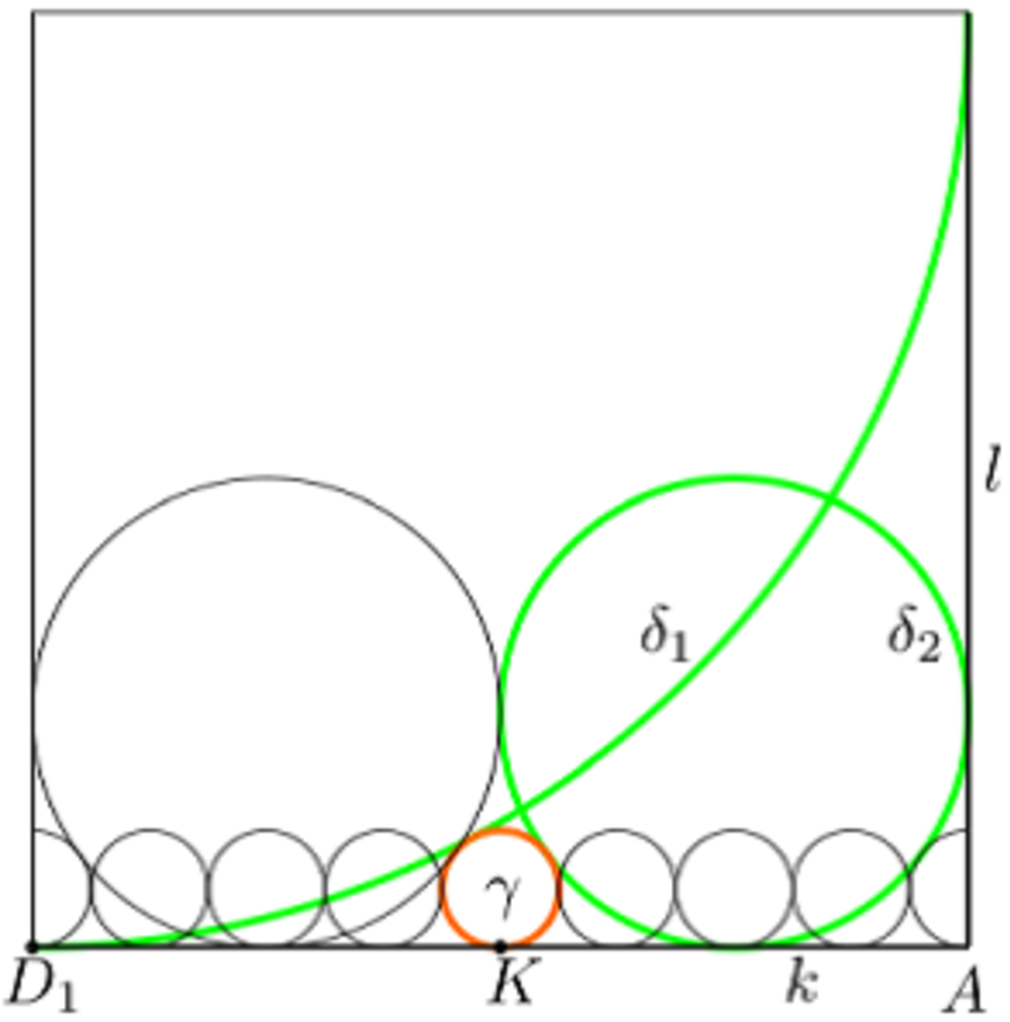}\refstepcounter{num}\label{f4half}\\
\vskip.7mm
Figure \Fg : $\CT(9/2)$
\end{center}  %%%%%%%%%%%%%%%%%%%%%%%%%@
\end{minipage}
\medskip

Notice that $K$ coincides with $A$ if and only if $n=1/2$ (see Figure 
\ref{fdefhalf}), also $1/2<n$ or $0<n<1/2$ according as $K$ and $\De_1$ 
lie on the same side of $l$ or not. If $n=0$, $\De_1$ coincides with $\De_2$ 
and is the reflection of $\G$ in $l$ (see Figure \ref{fdef0}). 
Our definition of $\CT(n)$ implies $n\ge0$ for any $n$. 
% including the case $n=\Ol{0}$. 
If $n$ is a natural number, there are circles $\G_1$, $\G_2$, $\cdots$, 
$\G_n$ of radius $r$ touching $k$ from the same side such that $\G=\G_1$, 
$\G_1$ and $\G_2$ touch, $\G_i$ $(i=3,4,\cdots n)$ touches $\G_{i-1}$ from 
the side opposite to $\G_1$, and $l$ is the external common tangent of 
$\G_n$ and $\De_1$ (see Figure \ref{f4}). This is the case considered by 
Toyoyoshi stated as Problem \ref{pbt}. %\cite{Toyo2}. 
If $n=2$, the circles $\G_2$ and $\De_2$ coincide (see Figure \ref{fc2}, 
where regard $\G_1=\G(2)$ in the figure). 
If we add the reflection of $\De_1$ in $l$ and remove $\De_2$ and $l$ for 
$\CT(n/2)$ in the case $n$ being a natural number, the resulting figure is 
a part of the figure $\CB(n)$ in \cite{OKSJM171}. i.e., $\CT(n)$ is 
a generalization of $\CB(n)$ in this sense. The first half of Theorem 
\ref{tg}(i) gives a solution of Problems \ref{ptsh}, \ref{pg} and \ref{pbt}. 

\begin{thm}\label{tg} 
The following statements are true for $\CT(n)$. \\
\noindent 
{\rm (i)} If $n\not=\Ol{0}$, 
%$\dfrac{1}{r}=\dfrac{\left(\sqrt{2n}-1\right)^2}{d_1}=\dfrac{\left(\sqrt{2n}+1\right)% ^2}{d_2}$. \\  
$d_1=\left(\sqrt{2n}+1\right)^2r$ 
and $d_2=\left(\sqrt{2n}-1\right)^2r$. \\
\noindent 
{\rm (ii)} $|AK|=\sqrt{d_1d_2}$. \\ 
\noindent 
{\rm (iii)} 
$  
\sqrt{r}=\left\{ 
\begin{array}{llll}\nonumber
\dfrac{\sqrt{d_1}+\sqrt{d_2}}{2} & \hbox{if $0\le n\le1/2$,} \\ 
\dfrac{\sqrt{d_1}-\sqrt{d_2}}{2} & \hbox{if $1/2<n$ or $n=\Ol{0}$.}  
\end{array} \right.
$  
\end{thm}

\begin{proof} If $n\not=\Ol{0}$, by Proposition \ref{p1} 
$d_2=\tau|AK|-2\sqrt{rd_2}=(2n-1)r-2\sqrt{rd_2}$ if $\tau=1$, and 
$d_2=\tau|AK|+2\sqrt{rd_2}=(2n-1)r+2\sqrt{rd_2}$ if $\tau=-1$, which 
yield $d_2=(\sqrt{2n}\pm1)^2r$ in both the cases. 
Also we have $d_1=\tau|AK|+2\sqrt{rd_1}=(2n-1)r+2\sqrt{rd_1}$, 
which also yields $d_1=(\sqrt{2n}\pm1)^2r$. This proves (i). 
The part (ii) is obvious if $n=\Ol{0}$. If $n\not=\Ol{0}$, (ii) 
follows from $|AK|=|2n-1|r$ and (i). The part (iii) is trivial if 
$n=\Ol{0}$. If $n\not=\Ol{0}$, eliminating $n$ from the two 
equations in (i) we get (iii). 
%If $0<n<1/2$, 
%$\sqrt{d_1/r}=1-\sqrt{2n}$, and $\sqrt{d_2/r}=1+\sqrt{2n}$. Hence
%$\sqrt{d_1/r}+\sqrt{d_2/r}=2$
%If $1/2<n$, 
%$\sqrt{d_1/r}=\sqrt{2n}-1$, and $\sqrt{d_2/r}=1+\sqrt{2n}$. Hence
%$\sqrt{d_2/r}-\sqrt{d_1/r}=2$
\end{proof}

If $n=4$, $\De_1$ and $\De_2$ intersect and the maximal circle 
touching $\De_1$ and $\De_2$ from inside of them has radius $r$, which 
is obtained by translating $\G_2$ parallel to $l$, through distance 
$4r$ (see Figure \ref{f4}). Let $D_i$ be the point of contact of 
$\De_i$ and $k$. If $n=9/2$, then $d_1=4d_2=16r$ and $K$ is the 
midpoint of $D_1A$ (see Figure \ref{f4half}). Problems considering 
this case with the circle $\De_2$ can be found in \cite{Toyo3, 
Toyo35}, \cite{Gazen, Gazen2} and \cite{zzkiou}. 
However the circle $\De_2$ seems to have been ignored for the figure 
$\CT(n)$ in most cases except this case. 
Let $E_i$ be the point of intersection of $k$ and the internal common 
tangent of $\De_i$ and $\G$, if $\De_i$ and $\G$ are proper circles 
(see Figure \ref{f63}). 

\begin{thm}\label{tccn} 
If $n\not=0, \Ol{0}$, the following statements hold for $\CT(n)$. \\ 
\noindent{\rm (i)}
The point $E_1$ divides $D_1A$ internally in the ratio $1:\sqrt{2n}$.\\
\noindent{\rm (ii)} 
If $n\not=1/2$, the point $E_2$ divides $D_2A$ externally in the ratio 
$1:\sqrt{2n}$. 
\end{thm}

\begin{proof} 
Let $n\not=0, \Ol{0}$. We prove (ii). Let $n\not=1/2$. Since $E_2$ is the 
midpoint of the segment $D_2K$, 
$|D_2E_2|=\sqrt{d_2r}=\left|\sqrt{2n}-1\right|r$ by Proposition \ref{p1} 
and Theorem \ref{tg}(i). 
If $1/2<n$, $D_2$ lies between $A$ and $E_2$. 
%(see Figure \ref{fdefcn}). 
Hence $|AE_2|=|D_2E_2|+d_2=
\left(\sqrt{2n}-1\right)r+\left(\sqrt{2n}-1\right)^2r=
\sqrt{2n}\left|\sqrt{2n}-1\right|r$. 
If $0<n<1/2$, $A$ lies between $D_2$ and $E_2$ (see Figure \ref{f63}). 
Therefore $|AE_2|=|D_2E_2|-d_2=
\left(1-\sqrt{2n}\right)r-\left(1-\sqrt{2n}\right)^2r=
\sqrt{2n}\left|\sqrt{2n}-1\right|r$. Hence we get 
$|D_2E_2|:|AE_2|=1:\sqrt{2n}$ in both the cases. 
The part (i) is proved in a similar way.
%$|D_2E_2|=\sqrt{d_2r}=\left|\sqrt{2n}+1\right|r$, 
%$|AE_2|=d_2-|D_2E_2|=
%\left(1+\sqrt{2n}\right)^2r-\left(1+\sqrt{2n}\right)r=
%\sqrt{2n}\left|\sqrt{2n}+1\right|r$
\end{proof}

\begin{center}%%%%%%%%%%%%%%%%%%%%%%%%%
\includegraphics[clip,width=65mm]{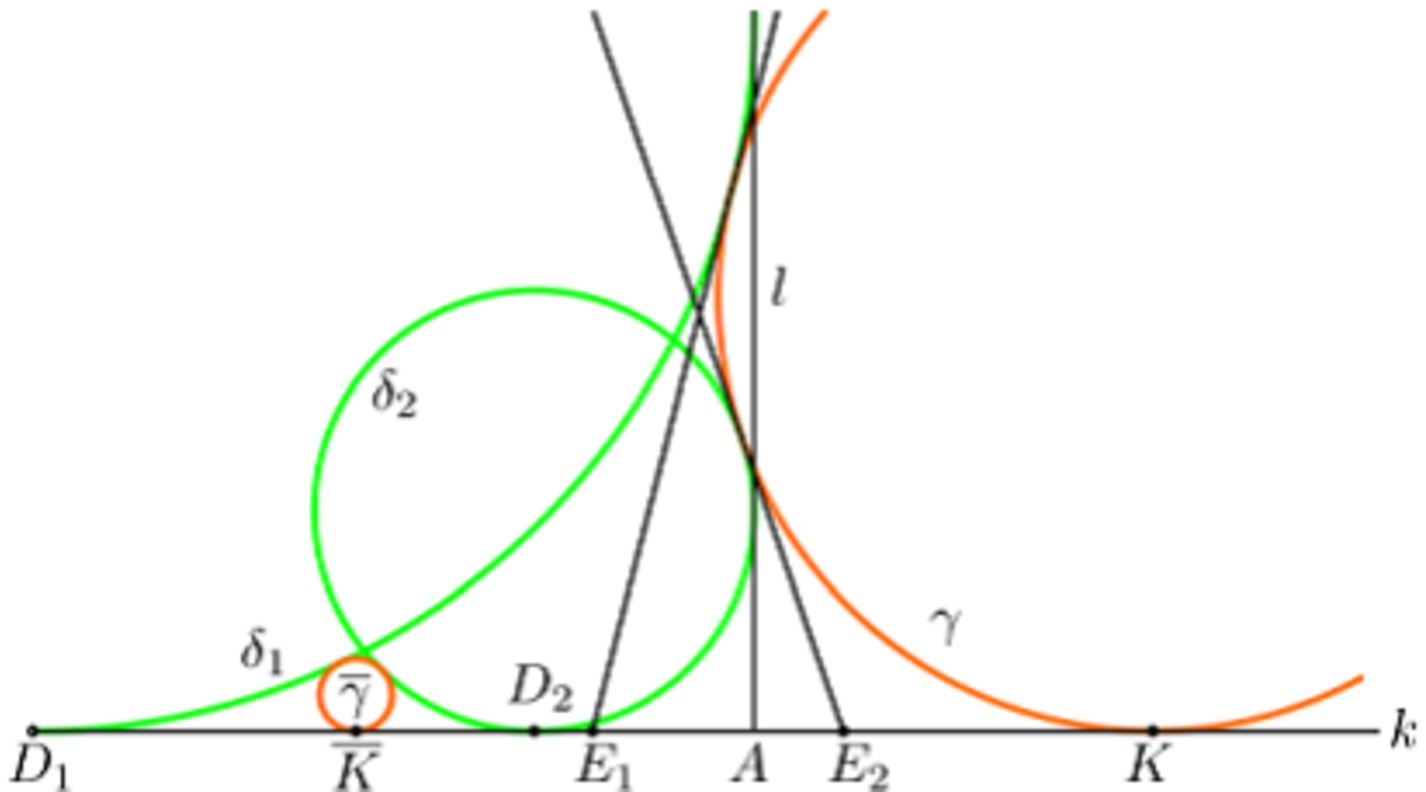}\refstepcounter{num}\label{f63}\\
\vskip4mm
Figure \Fg : $\CT(n)$ and $\CT(\Ol{n})$ in the case $0<n<1/2$
\end{center}%%%%%%%%%%%%%%%%%%%%%%%%%

\medskip
Let $\Ol{\G}=\G$ if $n=1/2$, $\Ol{\G}=D_1$ if $n=0$, $\Ol{\G}$ be 
the reflection of $\De_1$ in $l$ if $n=\Ol{0}$, and let 
$\Ol{\G}$ be the remaining circle touching $k$ and $\De_1$, $\De_2$ 
externally if $n\not=1/2, 0, \Ol{0}$ for the figure $\CT(n)$. 
We denote the figure consisting of the circles $\De_1$, $\De_2$, 
$\Ol{\G}$ and the lines $k$ and $l$ by $\CT(\Ol{n})$ (see Figure 
\ref{f63}). Notice that $\Ol{\Ol{\G}}=\G$, 
$\CT\left(\Ol{\Ol{n}}\right)=\CT(n)$, 
$\CT\left(\Ol{1/2}\right)=\CT(1/2)$. %and  
%$\CT\left(\Ol{0}\right)=\CT(\Ol{0})$. 
%$n=1/2$ and $\Ol{n}=1/2$ are equivalent 
(see Figures \ref{fdefinf} to \ref{fdefhalf}).  
%$n=0$ (resp. $n=\Ol{0}$) if and only if $\Ol{n}=\Ol{0}$ (resp. $n=0$) 
Let $\Ol{K}$ be the point of contact of $\Ol{\G}$ and $k$, and 
let $\Ol{r}$ be the radius of $\Ol{\G}$. 
%Let $\Ol{E_i}$ be the point of intersection of $k$ and the internal common 
%tangent of the circles of $\De_i$ and $\Ol{\G}$ if they are proper circles. 
%Let $E_1=D_1=A$ if $n=1/2$, and let $E_1=E_2=K=D_1$ if $n=\Ol{0}$. 
We now assume the definition of the division by zero \cite{kmsy}, i.e., 
$n/0=0$ for any real number $n$. Hence we also get $n/\Ol{0}=0$.

\begin{thm}\label{ttwo} The following statements hold for $\CT(n)$ 
and $\CT(\overline{n})$. \\
\noindent{\rm (i)}  $|AK|=|A\Ol{K}|$. \hskip32.5mm 
\noindent{\rm (ii)} $2n=\dfrac{1}{2\Ol{n}}$. \\
\noindent{\rm (iii)} $2(r+\Ol{r})=d_1+d_2$. \hskip21.7mm
%\noindent{\rm (iv)} $|E_iK|=|D_iE_i|=\left|A\Ol{E_i}\right|$ for $i=1,2$.\\ 
%\noindent{\rm (v)} $|E_1E_2|=\left|\Ol{E_1}\hskip.5mm\Ol{E_2}\right|$ 
%and $\left|E_1\Ol{E_1}\right|=\left|E_2\Ol{E_2}\right|$.
\end{thm}

\begin{proof} The part (i) follows from Theorem \ref{tg}(ii). 
We prove (ii). Since $n=0$ and $\Ol{n}=\Ol{0}$ are equivalent, and 
$n=\Ol{0}$ and $\Ol{n}=0$ are also equivalent, both side of (ii) equal 
0 if $n=0$ or $n=\Ol{0}$ by the definition of the division by zero. 
Hence (ii) holds in this case. Let $n\not=0, \Ol{0}$. The case $n=1/2$ 
is trivial. If $1/2<n$, the points $K$, $D_2$, $A$, $\Ol{K}$ lie in 
this order. Hence $2\sqrt{d_2r}=|D_2K|<|D_2\Ol{K}|=2\sqrt{d_2\Ol{r}}$ 
by (i), i.e., $r<\Ol{r}$. Therefore $r=(\sqrt{d_1}-\sqrt{d_2})^2/4$ and 
$\Ol{r}=(\sqrt{d_1}+\sqrt{d_2})^2/4$ by Theorem \ref{tg}(iii). Hence 
$$
n\Ol{n}=\frac{|AK|+r}{2r}\cdot \frac{-|A\Ol{K}|+\Ol{r}}{2\Ol{r}}=
\frac{\left(\sqrt{d_1}+\sqrt{d_2}\right)^2}{2\left(\sqrt{d_1}-\sqrt{d_2}\right)^2} 
\cdot 
\frac{\left(\sqrt{d_1}-\sqrt{d_2}\right)^2}{2\left(\sqrt{d_1}+\sqrt{d_2}\right)^2}=\frac{1}{4}. 
$$
The rest of (ii) is proved similarly. The part (iii) follows from 
Thereom \ref{tg}(iii). 
%The part (iv) is trivial if $n=1/2, 0, \Ol{0}$. 
%In the other case, the point $\Ol{E_1}$ divides $D_1A$ 
%externally in the ratio $1:\sqrt{2\Ol{n}}=\sqrt{2n}:1$ by Theorem 
%\ref{tccn}(i) and (ii), i.e., $|E_1K|=|D_1E_1|=\left|A\Ol{E_1}\right|$. 
%The rest of (iv) is proved similarly. The part (v) is trivial if 
%$n=1/2, 0, \Ol{0}$. If $0<n<1/2$, $A$ lies between $E_1$ and $E_2$. Therefore
%%%$|D_1K|-|E_2K|=2\sqrt{d_1r}-\sqrt{d_2r}=(2(1-\sqrt{2n})-(\sqrt{2n}+1))r
%%%=(1-\sqrt{2n})r>0$. Therefore $E_2$ lies between $D_1$ and 
%%%$K$, i.e., $D_1$ lies between $\Ol{E_1}$ and $E_2$. Hence 
%$|E_1E_2|=\left|AE_1\right|+|AE_2|=
%\left|D_1\Ol{E_1}\right|+\left|D_2\Ol{E_2}\right|=
%\left|\Ol{E_1}\hskip.5mm\Ol{K}\right|+\left|\Ol{E_2}\hskip.5mm\Ol{K}\right|=
%\left|\Ol{E_1}\hskip.5mm\Ol{E_2}\right|$ by (iv). 
%The rest of (v) is obvious.   
\end{proof}

\begin{remark}{\rm 
One may think that the equation in (ii) should be expressed as 
$4n\Ol{n}=1$. But this does not hold in the case $n=0$ or $n=\Ol{0}$. 
However, the equation $2n=\frac{1}{2\Ol{n}}$ holds even in this case by 
the definition of the value $\Ol{0}$ and the definition of the division 
by zero.}
\end{remark}

%%%%%%%%%%%%%%%%%%%%%%%%%%%%%%%%%%%%%%%%%%%%%%%%%%%%%%%%%%%%%%%%%%%%%%%%%
%%%%%%%%%%%%%%%%%%%%%%%%%%%%%%%%%%%%%%%%%%%%%%%%%%%%%%%%%%%%%%%%%%%%%%%%%
\section{Parametric representation of the generalized Haga's fold}\label{srela}

Let $ABCD$ be a square with a point $E$ on the line $DA$. 
% such that $D$ is not the midpoint of $AE$. 
We assume that $m$ is the perpendicular bisector of the segment $CE$, $G$ 
is the reflection of $B$ in $m$, and $F$ is the point of intersection of 
the lines $AB$ and $EG$ if they meet, where we define $F=B$ in the case 
$E=A$. The figure consisting of $ABCD$ and the points $E$, $F$ (if 
exists) and $G$ is called the figure made by the generalized Haga's 
folds 
and denoted by $\CH$ \cite{OKSJM174}. Ordinary Haga's fold is obtained if 
$E$ lies between $D$ and $A$. Let $\De$ be the circle of radius $d=|AB|$ 
with center $C$. In this section  we give a parametric representation of 
$\CH$ using the circle touching the line $DA$ and the circle $\De$ 
externally. 

Let $\G$ be the circle touching the line $DA$ and $\De$ at the point of 
contact of $\De$ and the remaining tangent of $\De$ from $E$. Then 
$\G\longmapsto\CH$ is %an injection from % 
one-to-one correspondence between the set of the circles touching the 
line $DA$ and $\De$ externally and the set of the figures made by the 
generalized Haga's fold, where we consider that the point $D$ is a 
member of the former set and $D$ corresponds to the figure made by 
the generalized Haga's fold with $D=E$. If $\G$ is a circle of radius 
$r$ touching $\De$ externally and the line $DA$ at a point $K$, we define 
$n$ by \eqref{eqndef}. % let $n=(\tau|AK|+r)/(2r)$, where $\tau=1$ if $\De$ 
%and $K$ lies on the same side of $AB$ otherwise $\tau=-1$. 
Let $T$ be the point of contact of 
$\G$ and $\De$. We explicitly denote the circle $\G$ by $\G(n)$ or 
$\G(-n)$ according as $T$ lies inside of $ABCD$ or outside of $ABCD$. If 
$T=B$, $\G$ is denoted by $\G(0)$. The point $D$ is denoted by 
$\G\left(\Ol{0}\right)$. Now any proper circle touching the line $DA$ and 
$\De$ externally can be expressed by $\G(n)$ for a real number $n$, and we 
also explicitly denote the figure $\CH$ by $\CH(n)$. If $D=E$, the figure 
is denoted by $\CH\left(\Ol{0}\right)$\footnote{$\CH\left(\Ol{0}\right)$ 
is denoted by $\CH(\infty)$ in \cite{OKSJM174}}. 

There are seven cases to be considered for $\CH$:  \\
(h1) $D$ lies between $E$ and $A$ and $b>d$, \\% (see Figure \ref{fcase1}),  
(h2) $D$ is the midpoint of $EA$, i.e., $b=d$, \\% (see Figure \ref{fcase2}) 
(h3) $D$ lies between $E$ and $A$ and $b<d$ \\%(see Figure \ref{fcase3}), 
(h4) $D=E$, \\% (see Figure \ref{fcase4}), \\
(h5) $E$ lies between $D$ and $A$, \\% (see Figure \ref{fcase5}), \\  
(h6) $E=A$, \\% (see Figure \ref{fcase6}), \\
(h7) $A$ lies between $D$ and $E$, 

%(see Figures \ref{fcase71}, \ref{fcase72}, \ref{fcase73}).

\medskip
\begin{minipage}{0.5\hsize} %%%%%%%%%%%%%%%%%%%%%%%%%
\begin{center} 
\includegraphics[clip,width=48mm]{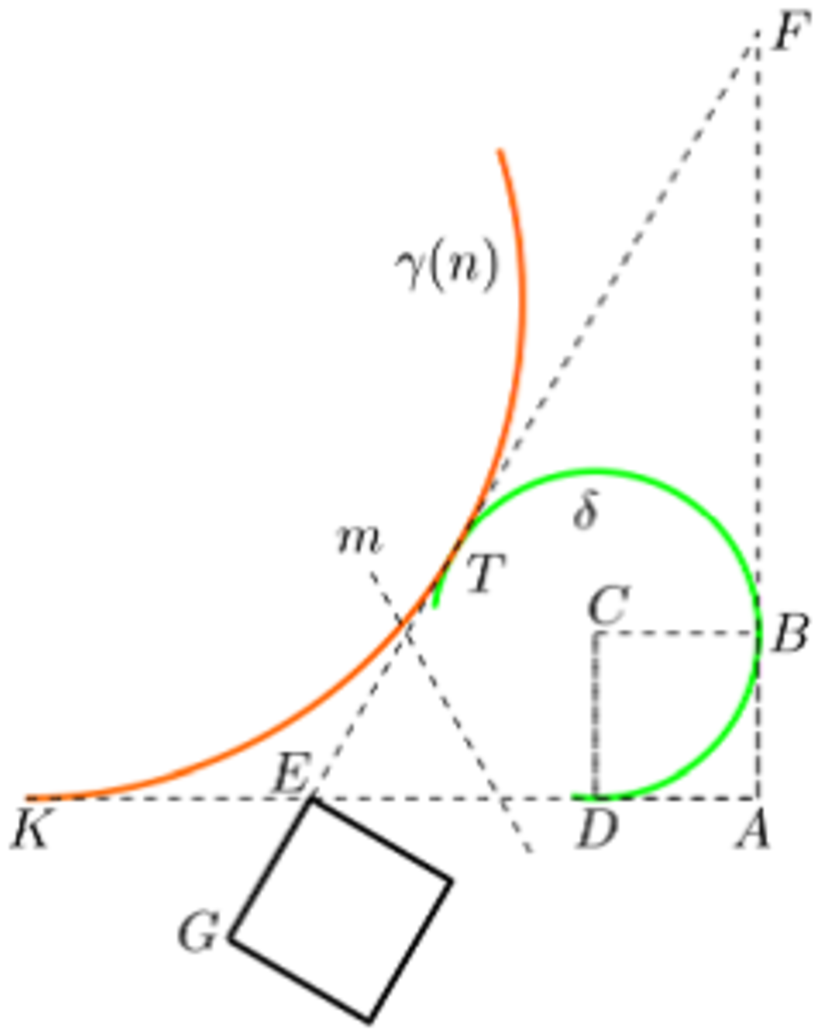}\refstepcounter{num}\label{fsep1}\\
Figure \Fg : (h1) $\CH(n)$, $(-2<n<-1/2)$
\end{center}  
\end{minipage}
\begin{minipage}{0.5\hsize}
\begin{center} 
\vskip24.6mm
\includegraphics[clip,width=48mm]{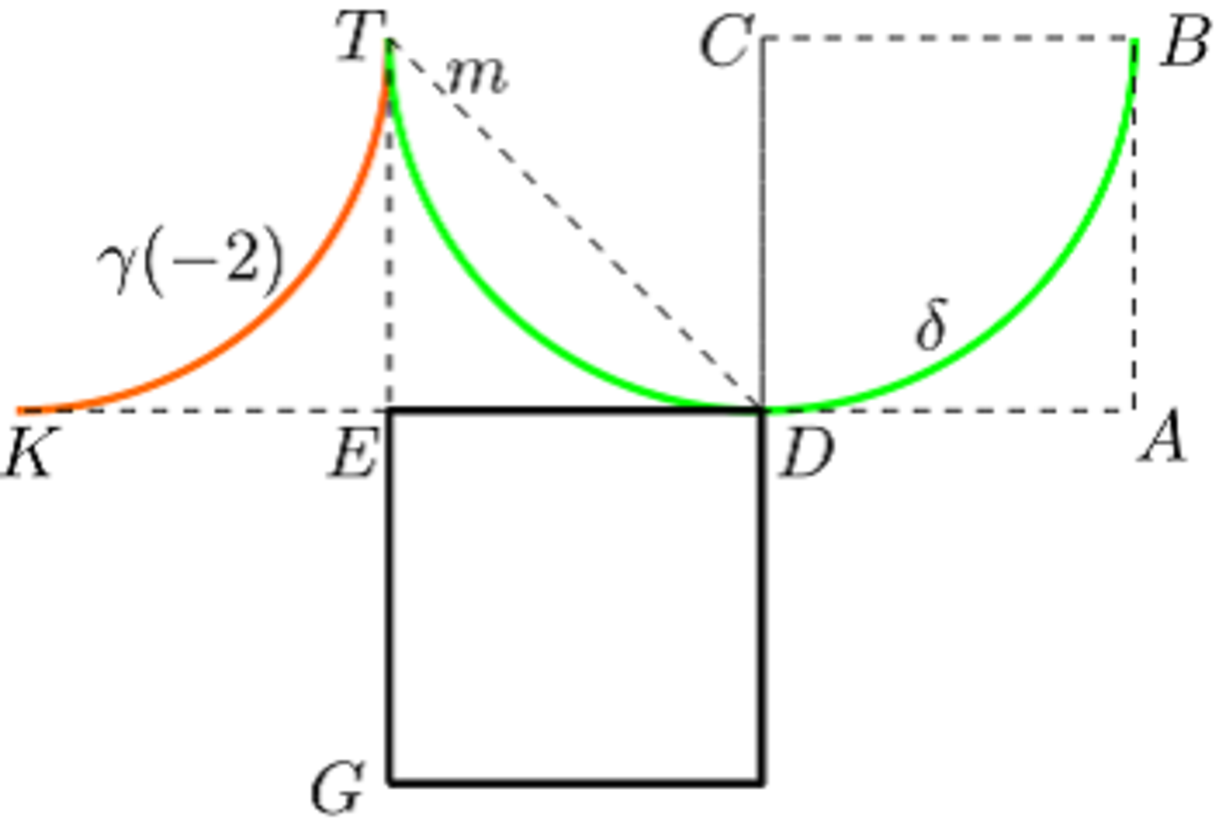}\refstepcounter{num}\label{fsep22}\\
Figure \Fg : (h2) $\CH(-2)$ 
\end{center}  
\end{minipage}%%%%%%%%%%%%%%%%%%%%%%%%%%%%%%%%%%%%%%%%%%%%
\medskip

\begin{minipage}{0.5\hsize}%%%%%%%%%%%%%%%%%%%%%%%%%%%%%%%
\begin{center} 
\vskip3mm
\includegraphics[clip,width=48mm]{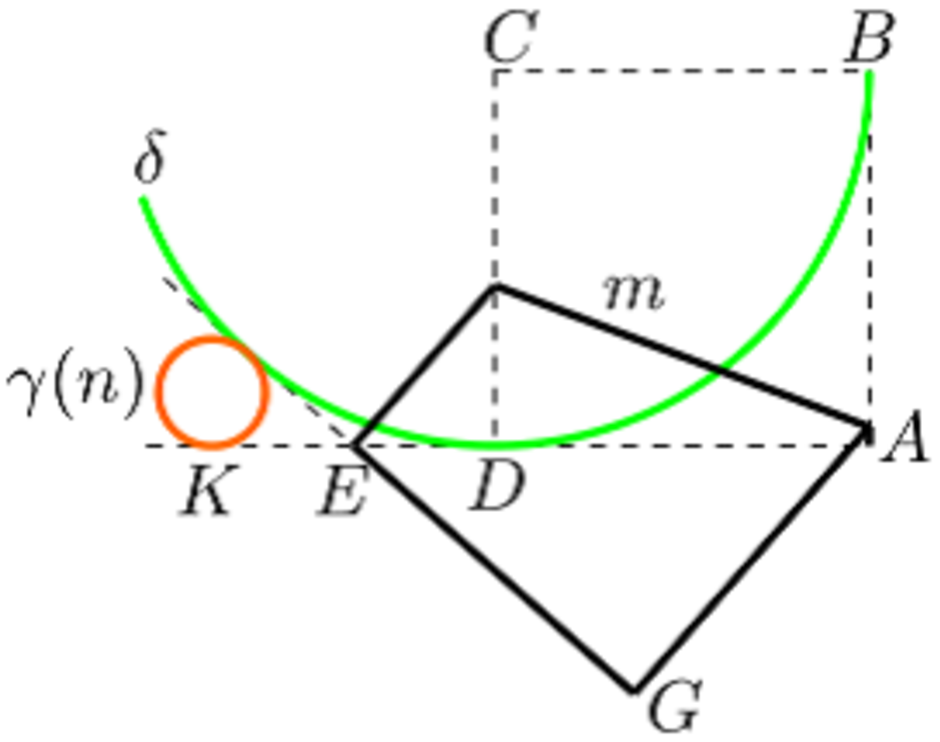}\refstepcounter{num}\label{fh1cc}\\
\vskip3mm
Figure \Fg : (h3) $\CH(n)$, $(n<-2)$ 
\end{center} 
\end{minipage} 
\begin{minipage}{0.5\hsize}
\begin{center} 
\includegraphics[clip,width=48mm]{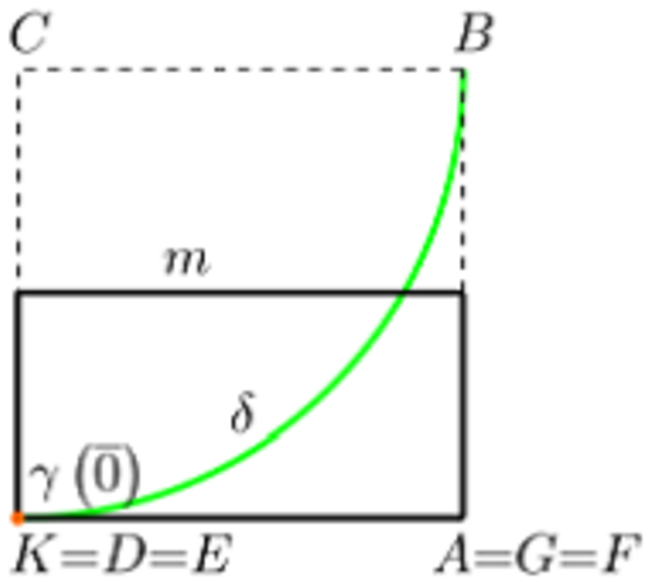}\refstepcounter{num}\label{finf}\\
%\input inf.TEX\refstepcounter{num}\label{finf}\\ 
%\vskip2mm
Figure \Fg :  (h4) $\CH\left(\Ol{0}\right)$
\end{center}  
\end{minipage}%%%%%%%%%%%%%%%%%%%%%%%%%%%%%%%%%%%%%%%%%%
\medskip

\medskip
\begin{minipage}{0.45\hsize}%%%%%%%%%%%%%%%%%%%%%%%%%
\begin{center} 
\vskip3.5mm
\includegraphics[clip,width=48mm]{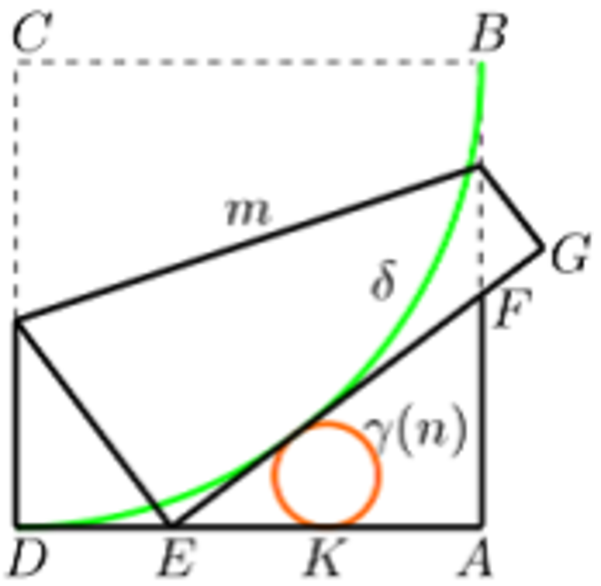}\refstepcounter{num}\label{fh5cc}\\
\vskip2.4mm
Figure \Fg : (h5) $\CH(n)$, $(0<n)$
\end{center}  
\end{minipage}
\begin{minipage}{0.55\hsize} 
\begin{center} 
\vskip1.6mm
\includegraphics[clip,width=57mm]{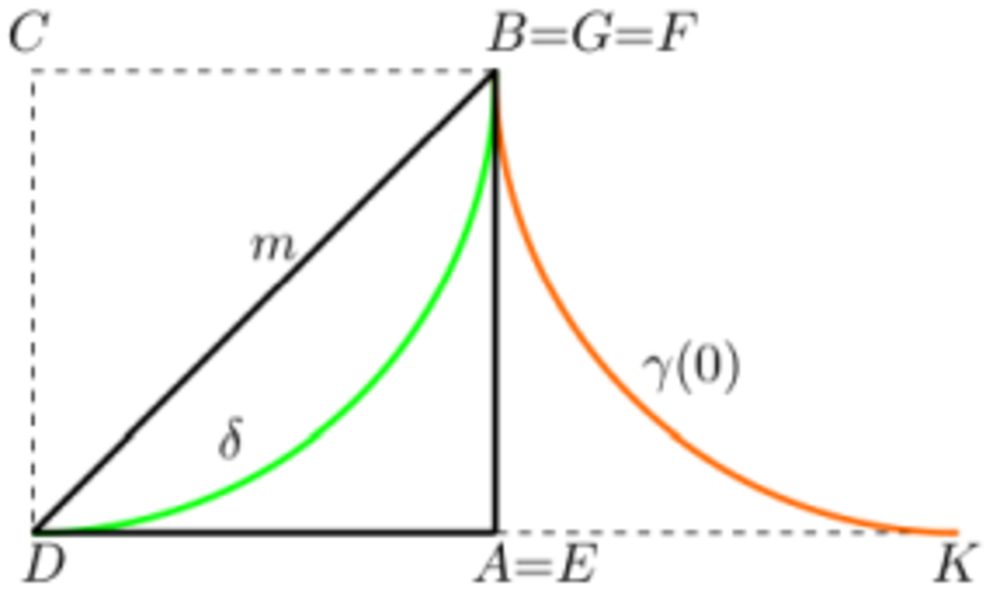}\refstepcounter{num}\label{f0}\\
\vskip-0.8mm
Figure \Fg : (h6), $\CH(0)$
\end{center}  
\end{minipage}%%%%%%%%%%%%%%%%%%%%%%%%%

\medskip
\begin{minipage}{0.53\hsize} %%%%%%%%%%%%%%%%%%%%%%%%%
\begin{center} 
\includegraphics[clip,width=48mm]{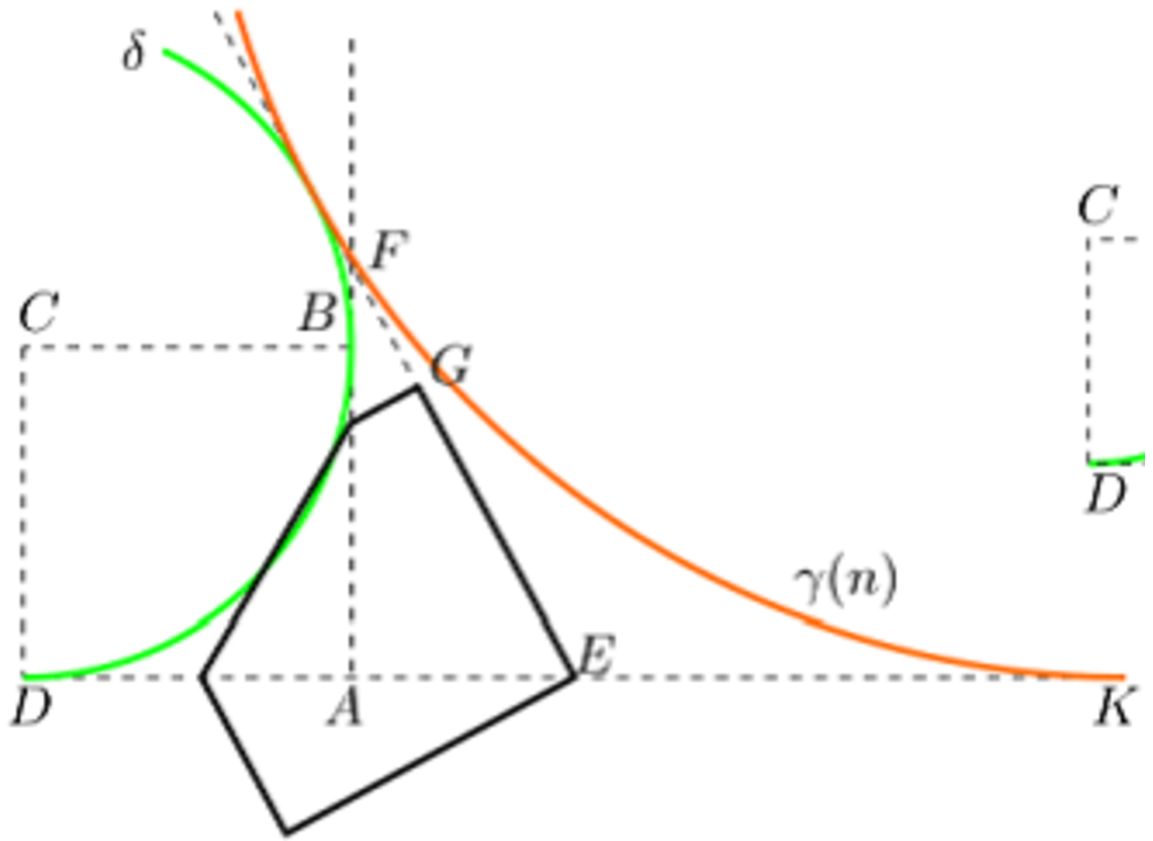}\refstepcounter{num}\label{fh3cc}\\
%\input h3cc.TEX\label{fh3cc}\refstepcounter{num}
%\vskip-5mm 
Figure \Fg : (h7) $\CH(n)$, $(-1/2<n<0)$ 
\end{center} 
\end{minipage} 
\begin{minipage}{0.42\hsize} 
\begin{center} 
\vskip4mm
\includegraphics[clip,width=58mm]{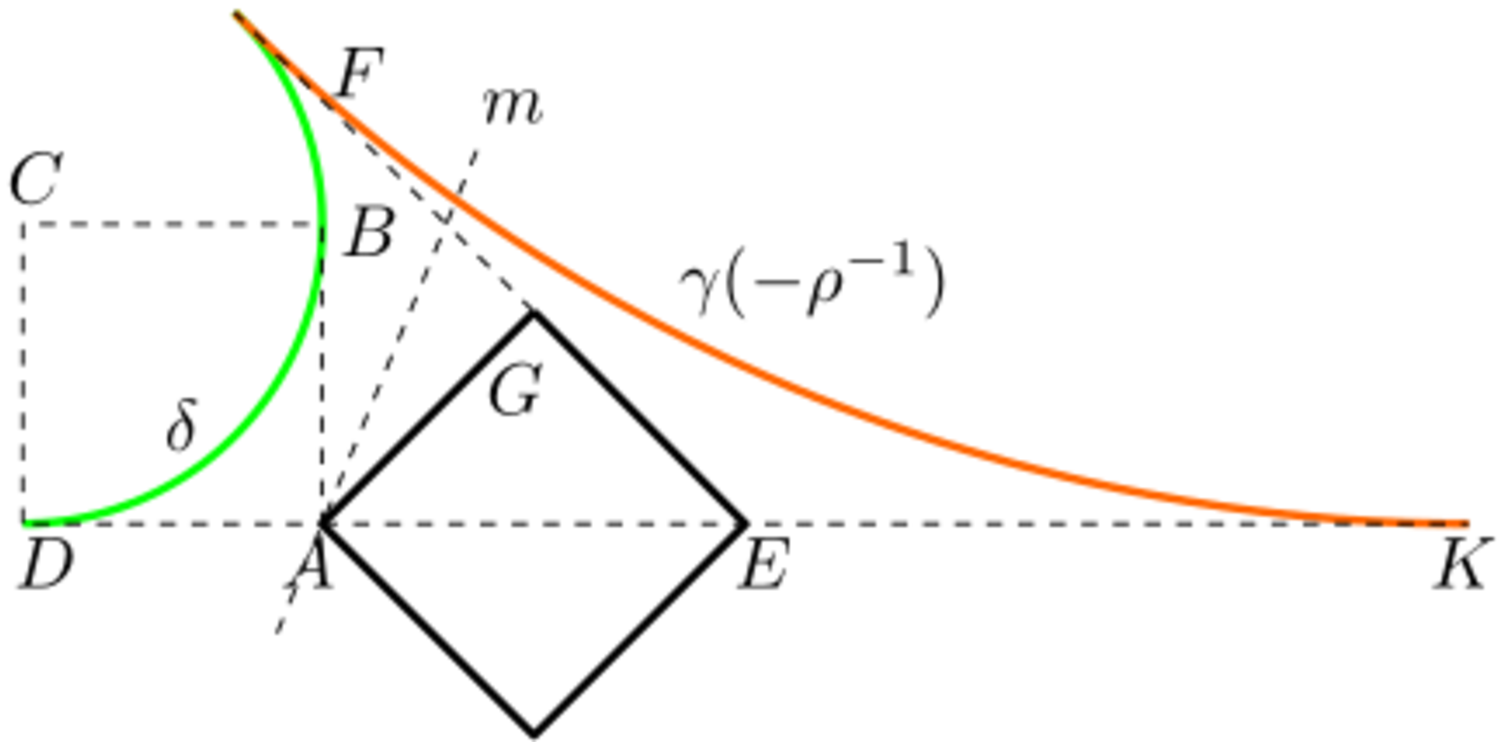}\refstepcounter{num}\label{fsep3}\\
\vskip8.8mm
Figure \Fg : (h7) $\CH(-\rho^{-1})$
\end{center}  
\end{minipage}%%%%%%%%%%%%%%%%%%%%%%%%%%%%%%%%%%%
\medskip

We consider the value of $n$ for $\CH(n)$ as a function of the point 
$E$ when $E$ moves on the line $DA$. We also specify the case in which 
the crease line $m$ passes through the inside of $ABCD$. In the cases 
(h1), (h2), (h3), the point $T$ lies outside of $ABCD$. Therefore $n<0$, 
and by Proposition \ref{p1} we have 
\begin{equation}\label{eqfunc1}
n=-\left(\frac{|AK|}{2r}+\frac{1}{2}\right)=-\frac{d+2\sqrt{dr}}{2r}-\frac{1}{2}
%=\sqrt{s}\frac{\sqrt{s}-2\sqrt{r}}{2r}+\frac{1}{2}
=-\sqrt{\frac{d}{r}}\left({\frac{1}{2}\sqrt\frac{d}{r}}+1\right)-\frac{1}{2}. 
\end{equation}
%While $dn/dr=\sqrt{d}(\sqrt{d}+\sqrt{r})/(2r^2)>0$. 
Therefore $n$ is a monotonically increasing function of $r$, i.e., $n$ 
decreases when $E$ 
moves with moving direction same as to $\Ov{DA}$. Hence $n$ approaches $-1/2$ 
when $D$ lies between $A$ and $E$ and $E$ moves away from $D$, i.e., $n<-1/2$. 
Therefore we get $-2<n<-1/2$, $n=-2$, $n<-2$ in the cases h(1), h(2), h(3), 
respectively. The crease line $m$ does not pass through the inside of $ABCD$ 
(see Figure \ref{fsep1}) in the case h(1). In the case (h2), $m$ passes 
through $D$ but does not pass through the inside of $ABCD$ (see Figure \ref{fsep22}). 
In the cases (h3), (h4), (h5), (h6), $m$ passes through the inside of $ABCD$ 
(see Figures \ref{fh1cc} to \ref{f0}), and we get $n=\Ol{0}$, $0<n$, $n =0$ in 
the cases (h4), h(5), (h6), respectively. The values of $n$ also decreases when 
$E$ moves from $D$ to $A$ in the case (h5). 

Let us consider the case (h7) (see Figure \ref{fh3cc}). Since $T$ lies 
outside of $ABCD$, we get 
\begin{equation}\label{eqfunc2}%\nonumber
n=-\left(-\frac{|AK|}{2r}+\frac{1}{2}\right)=\frac{2\sqrt{dr}-d}{2r}-\frac{1}{2}
%=\sqrt{s}\frac{\sqrt{d}-2\sqrt{r}}{2r}+\frac{1}{2}
=-\sqrt{\frac{d}{r}}\left({\frac{1}{2}\sqrt\frac{d}{r}}-1\right)-\frac{1}{2}. 
\end{equation}
%While $d<r$ implies $dn/dr=\sqrt{d}\left(\sqrt{d}-\sqrt{r}\right)/(2r^2)<0$. 
Hence $n$ is a monotonically decreasing function of $r$. Therefore $n$ 
approaches $-1/2$ when $r$ increases to $+\infty$, i.e., $n$ decreases 
when $E$ moves away from $A$, and we have $-1/2<n<0$. 

Let $\rho=\left(1+\sqrt{2}\right)^2$. 
We consider the special case in which $m$ passes through the point $A$. 
This happens if and only if $|AE|=\sqrt{2}d$, or $AEF$ is an isosceles 
right triangle (see Figure \ref{fsep3}). In this event 
$|DE|=|DK|/2=\sqrt{dr}=d+\sqrt{2}d$ holds by Proposition \ref{p1}. 
The equation implies $d/r=\rho^{-1}$, and we get $n=-\rho^{-1}$ by 
\eqref{eqfunc2}. Therefore {\it $m$ passes through the inside of $ABCD$ if 
and only if $-\rho^{-1}<n<0$ in the case} (h7). 

We summarize the results (see Table 1 also). The arrows in the table show 
that $n$ is a monotonically decreasing function of $E$ when $E$ moves on the 
line $DA$ with moving direction same as to $\Ov{DA}$. {\it The crease line 
$m$ passes through the inside of $ABCD$ in the cases 
{\rm (h3), (h4), (h5), (h6)}, and {\rm (h7)} if $-\rho^{-1}<n<0$.} 

\medskip
\begin{center}%%%%%%%%%%%%%%%%%%%%%%%%%%%%%%%%%%%%%%%%%%%%%
\begin{tabular}{|c|c|c|c|c|c|c|c|}
\hline 
case & (h1)      &(h2)&(h3)     &(h4)    &(h5)      &(h6)& (h7)     \\\hline
$n$&$-2<n<-1/2$&$-2$&$n<-2$&\hbox{\raisebox{-.3ex}[0ex][0ex]{$\Ol{0}$}}&$0<n$     &$0$ &$-1/2<n<0$\\\hline
     &$\searrow$ &    &$\searrow$&       &$\searrow$&    &$\searrow$\\
\hline
\end{tabular} \\
\medskip
Table 1.
\end{center}%%%%%%%%%%%%%%%%%%%%%%%%%%%%%%%%%%%%%%%%%%%

The above observation shows that $n\not=-1/2$, while the remaining tangent 
of the circle $\De$ parallel to $DA$ is not a member of the set of circles 
touching the line $DA$ and $\De$ externally. The fact suggests to 
describe the tangent by $\G(-1/2)$.

%%%%%%%%%%%%%%%%%%%%%%%%%%%%%%%%%%%%%%%%%%%%%%%%%%%%%%%%%%%%%%%%%%%%%%%%%%%%
\section{Special cases}

In this section we consider special cases for the figures $\CH(n)$. At first 
we consider two special cases of Haga's fold in the case 
(h5). If $E$ is the midpoint of $DA$, $F$ divides $AB$ in the ratio $2:1$ 
internally \cite[Theorem 3.1]{OKSJM174}. The fact is called Haga's first 
theorem \cite{koshiro}. While Theorem \ref{tccn}(i) shows that this happens 
if and only if $n=1/2$. Therefore the figure of Haga's first theorem 
coincides with $\CH(1/2)$ (see Figure \ref{fchalf}). A problem considering 
$\CH(1/2)$ giving the relation $d=4r$ can be found in \cite{kim}. 

Similarly if $F$ is the midpoint of $AB$, $E$ divides $DA$ in the ratio 
$1:2$ internally. The fact is called Haga's third theorem \cite{koshiro}. 
While Theorem \ref{tccn}(i) shows that this happens if and only if $n=2$. 
Hence the figure of Haga's third theorem coincides with $\CH(2)$ 
(see Figure \ref{fc2}). Since $E$ is the midpoint of the segment $DK$, 
$E$ and $K$ are the points of trisection of the side $DA$. 
%If we consider that $\De_2$ is the circle congruent to $\G(2)$ touching 
%$DA$, $AB$ and $\G(2)$ externally, and $k=DA$, $l=AB$, $\De_1=\De$, 
%they form the figure $\CT(2)$. 
The remaining circle touching $DA$ and $\De_1$ and $\De_2$ externally is 
$\G(1/8)$ by Theorem \ref{ttwo}(ii). Theorem \ref{tccn}(ii) shows that $K$ 
coincides with the point of intersection of $DA$ and the internal common 
tangent of $\G(1/8)$ and $\De$. Notice that $\CH(2)$ seems to be most 
frequently considered among $\CH(n)$ in Wasan geometry as we have shown 
in section \ref{srw}. 

\medskip
\begin{minipage}{.42 \hsize} %%%%%%%%%%%%%%%%%%%%%%%%%
\begin{center} 
\vskip5.6mm
\includegraphics[clip,width=48mm]{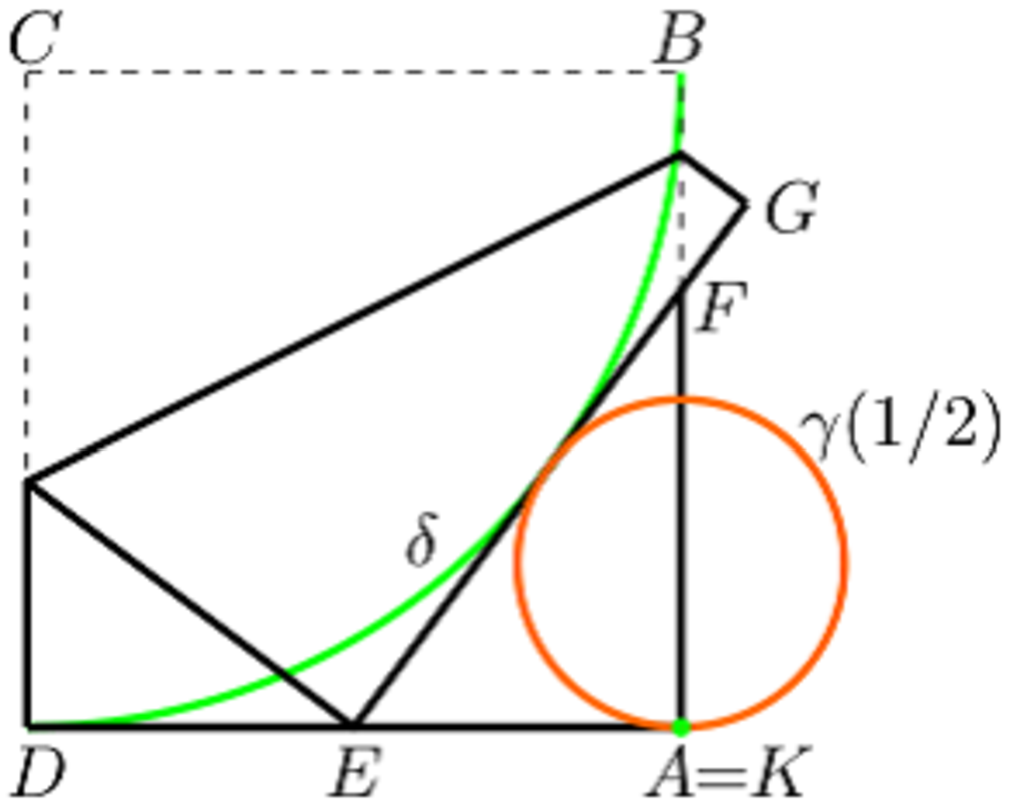}\refstepcounter{num}\label{fchalf}\\
\vskip2.6mm
Figure \Fg : $\CH(1/2)$
\end{center}  
\end{minipage}
\begin{minipage}{.58 \hsize}
\begin{center} 
\vskip0mm
\includegraphics[clip,width=48mm]{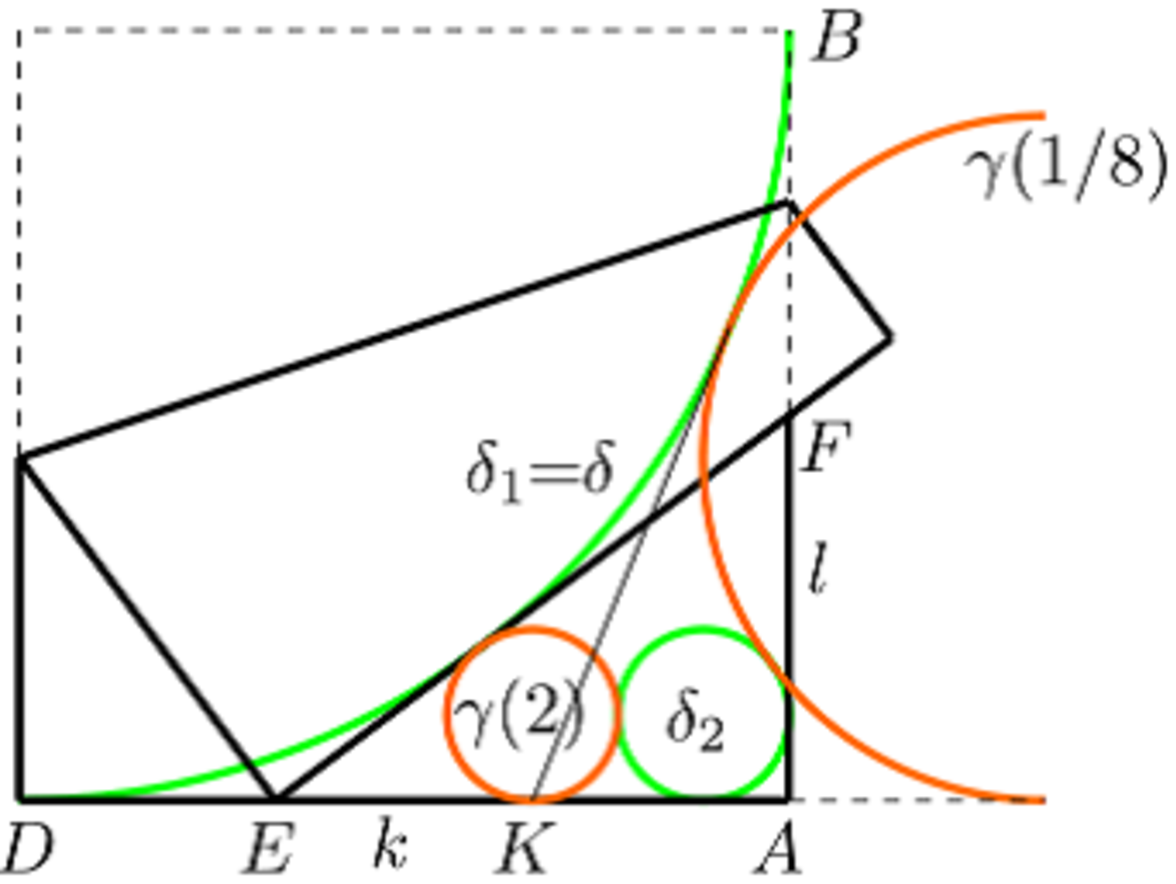}\refstepcounter{num}\label{fc2}\\
\vskip2mm
Figure \Fg : $\CH(2)$ with $\G(1/8)$
\end{center} 
\end{minipage} %%%%%%%%%%%%%%%%%%%%%%%%%@
\medskip

%%%%%%%%%%%%%%%%%%%%%%%%%%%%%%%%%%%%%%%%%%%%%%%%%%%%%%%%%%%%%%%
\section{Conclusion}

We argued the merit of considering circles in the geometry of origami 
in \cite{OKSCS16, OKK16}. In this two-part paper we have shown several 
examples to verify the validity of our assertion. The circles we have 
considered are 
tangent circles except the circumcircle of a triangle considered in the 
first part of the paper. In this sense we can say that the geometry of 
origami is a geometry of tangent circles. 
In particular, the incircle and the excircles of a right triangle play 
important roles in the geometry of origami using a square piece of paper
as shown in the first part of this paper.

%and the validity and usefulness of considering extended sides of a square. 

%%%%%%%%%%%%%

%%%%%%%%%%%%%%%%%%%%

%\medskip
%Hiroshi Okumura \\
%Maebashi Gunma 371-0123, Japan \\
%e-mail: \href{mailto:okmr@protonmail.com}{okmr@protonmail.com} \\


\begin{thebibliography}{99} 

\bibitem{Furu} Furuya ed., Samp\=o Ts\=usho, 
1854, Tohoku Univ. WDB, \\
\url{http://www.i-repository.net/il/meta_pub/G0000398wasan_4100005506}


\bibitem{gunma} Gunmaken Wasan Kenky\=ukai ed.,
The Sangaku in Gunma, Gunmaken Wasan Kenky\=ukai ,1987. 

 

%\bibitem{hansen} D. W. Hansen, On the radii of inscribed and escribed circles 
%of right triangles, The mathematics teacher {\textbf 72}(6) (1979) 462--464. 

%\bibitem{Hons} R. Honsberger. More Mathematical Morsels, MAA, 1991. 

\bibitem{Ishida} Ishida et al. ed., Mish\=o Samp\=o 
Vol. 10, 
Tohoku Univ. WDB, \\
\url{http://www.i-repository.net/il/meta_pub/G0000398wasan_4100007110} 

\bibitem{ito} It\=o ed., Samp\=o Zatsush\=o Senmonkai, Tohoku Univ. WDB, \\
\url{http://www.i-repository.net/il/meta_pub/G0000398wasan_4100005302}. 


\bibitem{kim} Kimura ed., Samp\=oki, 
Tohoku Univ. WDB, \\
\url{http://www.i-repository.net/il/meta_pub/G0000398wasan_4100005210}

\bibitem{koshiro} Koshiro, How to divide the side of square paper, \\
\url{http://www.origami.gr.jp/Archives/People/CAGE_/divide/02-e.html}  

\bibitem{kmsy}
M. Kuroda, H. Michiwaki, S. Saitoh and M. Yamane,
New meanings of the division by zero and interpretations on $100/0=0$ and 
on $0/0=0$, Int. J. Appl. Math. {\bf 27}(2) (2014), 191-198. 

\bibitem{OKSJM174} H. Okumura, Haga's theorems in paper folding and related
theorems in Wasan geometry Part 1, Sangaku J. Math., (2017) 42--56. 

\bibitem{OKSJM171} H. Okumura, Configurations of congruent circles on a line, 
Sangaku J. Math., (2017) 24--34. 

\bibitem{OKSCS16} H. Okumura, Origamics with Wasan geometry, 
Symmetry: Culture and Science, \textbf{28}(3) (2016) 312--320.

\bibitem{OKK16} H. Okumura, Origamics involving circles, 
RIMS K\=oky\=uroku, volume 1982 (2016) 37--42.

%\bibitem{OKMM14} H. Okumura, A Folded Square Sangaku Problem, 
%Mathematical Medley, \textbf{40}(2) (2014)  21--23. 

%\bibitem{OKFG14} H. Okumura, A note on Haga's theorems in paper folding, 
%Forum Geom., \textbf{14} (2014) 241--242. 

\bibitem{OKMIQ94} H. Okumura, 
Concentric figures with tangent circles, 
Mathematics and Informatics Quarterly, \textbf{4}(2) (1994) 87--91. 

%\bibitem{OKGZ90} H. Okumura, Incircles and excircles of right-angled 
%triangles, Mathematical Gazette, \textbf{74} No. 469 (1990) 278--279. 


\bibitem{Toyo1} Toyoyoshi  ed., Santei Y\=odai Katsuy\=o, 
Digital Library Department of Mathematics, Kyoto University, 
\url{http://edb.math.kyoto-u.ac.jp/wasan/068} 


\bibitem{Toyo2} Toyoyoshi  ed., Chikusaku, 
Digital Library Department of Mathematics, Kyoto University, 
\url{http://edb.math.kyoto-u.ac.jp/wasan/154}. 

\bibitem{Toyo3} Toyoyoshi , Tenzan, 
Digital Library Department of Mathematics, Kyoto University, 
\url{http://edb.math.kyoto-u.ac.jp/wasan/159}. 

\bibitem{Toyo35} Toyoyoshi  ed., Tenzan Zatsumon, 
Digital Library Department of Mathematics, Kyoto University, 
\url{http://edb.math.kyoto-u.ac.jp/wasan/161}. 

\bibitem{Toyo4} Toyoyoshi  ed., Tenzan Sany\=o, 
Digital Library Department of Mathematics, Kyoto University, 
\url{http://edb.math.kyoto-u.ac.jp/wasan/162}. 

\bibitem{Gazen} Yamamoto ed., Samp\=o Tenzan Tebiki Gusa, 1833, Tohoku Univ. WDB, 
\url{http://www.i-repository.net/il/meta_pub/G0000398wasan_4100005576}. 

\bibitem{Gazen2} Yamamoto ed., Taizen Jink\=oki, 1832, Tohoku Univ. WDB, \\
\url{http://www.i-repository.net/il/meta_pub/G0000398wasan_4100000422}. 

\bibitem{zzkiou} Yasuda et al. ed., Z\=ozoku Ki\=osh\=u, Tohoku Univ. WDB, \\ 
\url{http://www.i-repository.net/il/meta_pub/G0000398wasan_4100004196}. 

\bibitem{ets} Enrui Tekit\=osh\=u, Tohoku Univ. WDB, \\
\url{http://www.i-repository.net/il/meta_pub/G0000398wasan_4100003918}. 


\bibitem{kan} Kansei Heishin Sanron, Tohoku Univ. WDB, \\
\url{http://www.i-repository.net/il/meta_pub/G0000398wasan_4100004185}. 

\bibitem{sd1678} Sandai Kenmonki, 
Tohoku Univ. WDB, \\
\url{http://www.i-repository.net/il/meta_pub/G0000398wasan_4100005039}. 


\bibitem{ssk} Sansoku, Tohoku Univ. WDB, \\
\url{http://www.i-repository.net/il/meta_pub/G0000398wasan_4100005018}. 

\bibitem{tz1322} Tenzan Kaitei, Tohoku Univ. WDB, \\
\url{http://www.i-repository.net/il/meta_pub/G0000398wasan_4100006710}. 


\bibitem{uk} Wasan Hensh\=u, Tohoku Univ. WDB, \\
\url{http://www.i-repository.net/il/meta_pub/G0000398wasan_4100007321}. 



\vskip1.5mm
\noindent\hskip-7mm
Tohoku Univ. WDB is short for Tohoku University Wasan Material Database.
\end{thebibliography}
\end{document}